\documentclass[a4paper, 10pt]{amsart}

\usepackage[square, numbers, comma]{natbib} 

\usepackage[usenames,dvipsnames,svgnames,table]{xcolor}
\usepackage{amsmath,amsthm,amssymb,amssymb,esint,verbatim,tabularx,graphicx,subcaption}
\usepackage{bbm}
\usepackage{fancyhdr}
\usepackage{enumerate}
\usepackage{amsmath}
\usepackage{fancyhdr}
\usepackage{epic}
\usepackage{pgf,tikz}
\usetikzlibrary{arrows}

\usepackage[utf8]{inputenc}
\usepackage{color}
\usepackage{hyperref}
\usepackage{verbatim}
\usepackage{pdfpages}
\usepackage{empheq}

\hypersetup{urlcolor=blue, colorlinks=true} 

\usepackage[inner=2.5cm,outer=2.5cm,bottom=4cm,top=4cm]{geometry}

\newtheorem{theorem}{Theorem}[section]

\newtheorem{lemma}[theorem]{Lemma}

\newtheorem{corollary}[theorem]{Corollary}
\newtheorem{definition}{Definition}[section]

\theoremstyle{remark}
\newtheorem{remark}[theorem]{Remark}
\theoremstyle{definition}

\numberwithin{equation}{section}

\newcommand{\R}{\ensuremath{\mathbb{R}}}
\newcommand{\N}{\ensuremath{\mathbb{N}}}

\newcommand{\Z}{\ensuremath{\mathbb{Z}}}

\newcommand{\dd}{\,\mathrm{d}}

\newcommand{\Grid}{\mathcal{G}_h}

\newcommand{\LL}{\mathcal{L}}

\begin{document}

\title[Asymptotic expansions and finite differences for the fractional $p$-Laplacian]{Higher-order asymptotic expansions and finite difference schemes for the fractional $p$-Laplacian}

\author[F.~del~Teso]{F\'elix del Teso}

\address[F. del Teso]{Departamento de Matem\'aticas, Universidad Aut\'onoma de Madrid, Ciudad Universitaria de Cantoblanco, 28049 Madrid, Spain} 

\email[]{fdelteso\@@{}ucm.es}

\urladdr{https://sites.google.com/view/felixdelteso}


\author[M.~Medina]{María Medina}

\address[M. Medina]{Departamento de Matem\'aticas, Universidad Aut\'onoma de Madrid, Ciudad Universitaria de Cantoblanco, 28049 Madrid, Spain}
\email[]{maria.medina@uam.es}
\urladdr{}

\author[P.~Ochoa]{Pablo Ochoa}

\address[P. Ochoa]{Facultad de Ingenier\'ia, Universidad Nacional de Cuyo-CONICET, 5500 Mendoza, Argentina}
\email[]{pablo.ochoa@ingenieria.uncuyo.edu.ar}

\address[]{}
\email[]{}
\urladdr{}


\begin{abstract}   

We propose a new asymptotic expansion for the fractional $p$-Laplacian with precise computations of the errors. Our approximation is shown to hold in the whole range $p\in(1,\infty)$ and $s\in(0,1)$, with errors that do not degenerate as $s\to1^-$. These are super-quadratic for a wide range of $p$ (better far from the zero gradient points), and optimal in most cases. One of the main ideas here is the fact that the singular part of the integral representation of the fractional $p$-Laplacian behaves like a local $p$-Laplacian with a weight correction. As a consequence of this, we also revisit a previous asymptotic expansion for the classical $p$-Laplacian, whose error orders were not known.

Based on the previous result, we propose monotone finite difference approximations of the fractional $p$-Laplacian with explicit weights and we obtain the error estimates. Finally, we introduce explicit finite difference schemes for the associated parabolic problem in $\R^d$ and show that it is stable, monotone and convergent in the context of viscosity solutions. An interesting feature is the fact that the stability condition improves with the regularity of the initial data.

\end{abstract}

\maketitle

\setcounter{tocdepth}{1}

\tableofcontents 

\section{Introduction}
Consider $J_p:\R\to\R$ given as $J_p(\xi):=|\xi|^{p-2}\xi$ for $p>1$. We define the fractional $p$-Laplacian operator as
$$-(-\Delta)^s_p\phi(x):=\,P.V.\int_{\R^d}J_p(\phi(x+y)-\phi(x))\frac{\dd y}{|y|^{d+sp}},$$
with $s\in(0,1)$. The usual definition of the fractional $p$-Laplacian carries a normalizing constant (depending on $p$, $s$, and $d$) in front of the integral (see for instance \cite{dTGoVa21}). This constant is irrelevant for our purposes so, for the sake of clarity, we omit it in the definition and the formulation of the results. 

The limit operator (up to suitable normalization constant) as $s\to1^-$ is the so-called $p$-Laplacian, defined as
\[
\Delta_p\phi(x):= \nabla\cdot \left(|\nabla \phi(x)|^{p-2}\nabla\phi(x)\right).
\]

The aim of this paper is to give asymptotic expansions with high order error estimates  and discretizations of these operators, apply them to the study of a parabolic problem. 

More precisely, we first revisit the mean value property \eqref{Mlocal} introduced in \cite{dTLi21,BuSq22} for the $p$-Laplacian, and prove uniform error estimates that, in the best escenario, are of quadratic order (see Theorem \ref{asympExpUnif} and Theorem \ref{asympExp}). Later we exploit the fact that the singular part of the fractional $p$-Laplacian behaves like a weighted $p$-Laplacian to introduce a new asymptotic expansion for this operator, which has even better orders of convergence than the local case (see Theorem \ref{asympExpUnif} and Theorem \ref{asympExp}). Actually these orders improve with $p$ and do not degenerate as $s\to1^-$ (the orders of the local case are recovered).

Once asymptotic expansions have been introduced, we use suitable quadrature rules for the resulting nonlinear integral operators to provide two monotone finite difference discretizations. Both of them have the same order of consistency, and one of them produce explicit weights (i.e. they do not rely on computing integrals with respect to measures that are typically not explicit). This is an advantage computationally speaking.

Finally, we exploit the fact that all our approximation operators satisfy the maximum principle (i.e. at maximum points the operator has a sign), to introduce a convergent numerical scheme for the associated nonlocal parabolic problem in the context of viscosity solutions. An interesting feature here is that, for regular enough initial data, our scheme is shown to be stable under a very mild CFL-type stability condition that is consistent with the linear case.

\subsection{Description of the results and their  context in the literature}
The topic of asymptotic expansion for $p$-Laplacian type operators has been intensively studied in the last decade. The first known result in this direction is the seminal work of Manfredi, Parviainen and Rossi \cite{MaPaRo10} where the following asymptotic formula for $\Delta_p^{\textup{N}}\phi:=|\nabla \phi|^{2-p}\Delta_p\phi$ (the so-called normalized $p$-Laplacian) is presented for $2<p\leq\infty$:
\begin{equation*}
\mathcal{A}^{p}_r\phi(x)= \frac{1}{r^2}\left(\frac{p-2}{2(p+d)}( \max_{y\in \overline{B_r}} \phi(x+ y) + \min_{y\in \overline{B_r}} \phi(x+ y)) + \frac{d+2}{p+d}\fint_{B_r} \phi(x+y)\dd y\right).
\end{equation*}
As it is shown in \cite[Exercise 3.7]{Lew20}, this approximation is of order $O(r)$ outside of zero gradient points. In this direction another asymptotic expansion, valid also in the case $1<p<2$ is presented in \cite{KaMaPa12}, again of order $O(r)$.

The results for the $p$-Laplacian are due to Bucur and Squassina \cite{BuSq22}, and del Teso and Lindgren \cite{dTLi21}. They introduce the expansions presented in \eqref{Mlocal}, and show them to be consistent, but with no orders of convergence. As a counterpart of the results of this paper, here we obtain the order $O(r^{\min\{\gamma,2\}})$ for all $\gamma<p-1$ away from zero gradient points if $p>1$ (see Theorem \ref{asympExp} and Remark \ref{rem:optim} for its optimality when $p\geq3$), and $O(r^{\min\{p-2,2\}})$ for all points if $p\geq 2$ (see Theorem \ref{asympExpUnif} and Remark \ref{rem:optim} for its optimality). Our results are also coherent with the case $p=2$ where the natural asymptotic expansion is the usual mean value formula, known to be of order $O(r^2)$.

For the fractional $p$-Laplacian  $-(-\Delta)^s_p$, the only known asymptotic expansion is the one given by Bucur and Squassina in \cite{BuSq22}:
\begin{equation}\label{eq:asexpBuSq}
\mathcal{A}_{r}^{s,p}\phi(x)=\int_{\R^d\setminus B_r} J_p(\phi(x+y)-\phi(x))\frac{\dd y}{|y|^{d+(p-2)s}(|y|^2-r^2)^s}.
\end{equation}
The authors proved this approximation to be of order $O(r^{2-2s})$. In \eqref{Mrs} we propose a new asymptotic expansion with a number of improved properties with respect to \eqref{eq:asexpBuSq}:
\begin{enumerate}[$\bullet$]
\item The measure in \eqref{Mrs} is bounded for every $r>0$, while the one in  \eqref{eq:asexpBuSq} is always unbounded.
\item  Consistency of \eqref{Mrs} also holds in the range $p\in(1,2)$ at nonzero gradient points.
\item The order of convergence of \eqref{Mrs} improves with $p$.
\item The order of convergence of \eqref{Mrs} does not degenerate as $s\to1^-$.
\item While the order of \eqref{eq:asexpBuSq} is always sub-quadratic, the order of \eqref{Mrs} is super-quadratic if $p\geq 3/(2-s)$ (in particular $p\geq3$ for all $s\in(0,1)$) at nonzero gradient points, and if $p\geq 4/(2-s)$ and $p>2$ (in particular $p\geq4$ for all $s\in(0,1)$) at every point. We refer to Remark \ref{rem:optim} for a discussion on the optimality. 
\end{enumerate}
The fractional framework when $p=2$ has not explicitly been treated in the literature. Our approach, whose main idea is described in Section \ref{sec:asymptoticexporder} is also valid in that case.

One of the main ideas here is the fact that the singular part of the integral representation of the fractional $p$-Laplacian behaves like a local $p$-Laplacian with a weight correction. This is related to ideas presented in \cite{IsNa10,And-etal10}.

Asymptotic expansions for other nonlinear fractional operators have been treated in the literature, for instance in \cite{BuSq22,dTEnLe22,Lew22}. 

The second part of the paper regards finite difference approximations of the fractional $p$-Laplacian.  
In the local case, monotone discretizations of $\Delta_p^{\textup{N}}$ were introduced for the first time by Oberman in \cite{Obe13}. Other discretizations have been proposed later on (see e.g. \cite{CoLeMa17, dTMaPa22}). The case $p=\infty$ has also been treated (see e.g. \cite{Ob05,BuCaRo22,WeSa22}). For $\Delta_p$ the only known finite difference discretization is given by del Teso and Lindgren in \cite{dTLi21} (without results on the order of convergence). The results of this paper also apply to this case, and show the orders of convergence.

In the nonlocal case, the results of this paper are the first finite difference discretizations for the fractional $p$-Laplacian. Other works have proposed discrete versions of the fractional $p$-Laplacian (see e.g. \cite{JuDiXi20,ChBi22}), however, no consistency results in the line of Theorem \ref{DiscretOrder} are known. 

The fractional framework when $p=2$ has been extensively studied in the literature. A simplified linear version was successfully developed in dimension $d=1$ by Huang and Oberman in \cite{HuOb14} (see \cite{dTEnJa18} for higher dimensions), obtaining orders $O(h^{4-2s})$. An alternative discretization was proposed by Ciaurr et-al in dimension $d=1$ in \cite{Ciaetal18}, showing convergence of order $O(h^{2-2s})$ in dimension $d=1$. This idea was revisited in \cite{dTEnJa18} obtaining orders $O(h^2)$ unconditionally on $s\in(0,1)$ and in any dimension.

In the last part of this paper, we propose an explicit numerical scheme for the parabolic equation
\begin{equation}\label{eq:parabfplap}
\partial_t u +(-\Delta)_p^su=f \quad \textup{in} \quad \R^d\times(0,T)
\end{equation}
using the discretization previously introduced. In particular we show that, for an initial condition with certain $a$-H\"older regularity, the scheme is convergent under a stability condition of the form
\begin{equation*}
h=O(r), \quad \text{and }\quad 
\tau  \lesssim \left\{ \begin{split}r^{ 2s+(s-a)(p-2)} \quad &\textup{if} \quad  a<\frac{sp}{p-1}\\
\frac{r^{a}}{|\log(r)|} \qquad \quad \, \, \,  &\textup{if} \quad  a\geq\frac{sp}{p-1},
\end{split}\right.
\end{equation*}
where $h$ and $\tau$ are the mesh parameters of the discretizations in space and time respectively.
This result is coherent with the results for $p=2$ and $s\in(0,1)$ (see e.g. \cite{Dro10,dTEnJa19}) and also the ones in the local case for  $p>2$, recently proposed by del Teso and Lindgren in \cite{dTLi22b}. Problem \eqref{eq:parabfplap} has been previously studied (cf. \cite{And-etal10,Vaz20,Vaz21,Vaz22}).  The literature regarding numerical methods for PDE problems related to the fractional $p$-Laplacian is scarce. To our best knowledge, the only known result is the recent work of  Borthagaray, Li and Nochetto \cite{BoWeNo22} where the Dirichlet problem is treated in the finite element setting. In that work, the discrete version of the problem is expressed in terms of integrals of certain interpolation basis with respect to a singular integral (and in general it cannot be explicitly computed even in dimension $d=1$ and needs of a quadrature rule that is used for computational reasons). A main advantage of our approach is that the weights of the discretization are explicit and can be directly used  without an extra approximation step.

\subsection*{Notation} For $\rho >0$ we define
$$B_\rho(x):=\{y\in\R^d:\,|x-y|<\rho\},\quad B_\rho:=B_\rho(0).$$
The constant $\omega_d$ will denote the volume of the ball of radius $1$ in $\R^d$. 
Given an open set $\Omega \subset \R^d$, we consider
$$C_b^k(\Omega):=\{\phi\in C^k(\Omega):\,|D^\kappa\phi|\in L^\infty(\overline{\Omega}) \mbox{ for all multiindex }\kappa,|\kappa|=0,1,\ldots, k\},$$
equipped with the norm
$$\|\phi\|_{C_b^k(\Omega)}:=\sum_{|\kappa|=0}^k\|D^\kappa\phi\|_{L^\infty(\overline{\Omega})}.$$
Given  functions $f$ and $g$, we recall that $f=O(g)$ if  there is $C>0$ such that $|f(x)|\leq C|g(x)|$ for all $x$. Also, $f=o_g(1)$  as $x \to x_0$ if 
$$\lim_{x \to x_0}\dfrac{|f(x)|}{|g(x)|}=0.$$We say that $f\asymp g$ if there are constants $C_1, C_2 > 0$ such that
$$C_1 g \leq f \leq C_2 g.$$
For $p >1$, 
$$J_p(\xi):=|\xi|^{p-2}\xi,\quad p>1.$$

Finally, $C$ will denote a universal constant that may change from line to line. When relevant, we will add subscripts to specify the dependence of certain parameters.

\section{Main results}
In this section we introduce the principal contributions of the article.
\subsection{Asymptotic expansions} Let $s\in(0,1)$ and $p>1$. We propose the following asymptotic expansion of the fractional $p$-Laplacian: 
\begin{equation}\label{Mrs}
\mathcal{M}_r^{s,p}[\phi](x):= \frac{(p+d)}{p(1-s)r^{d+sp}}\int_{B_r}J_p(\phi(x+y)-\phi(x)) \dd y + \int_{\R^d\setminus B_r}J_p(\phi(x+y)-\phi(x)) \frac{\dd y}{|y|^{d+sp}}.
\end{equation}
We also recall the two asymptotic expansions for the $p$-Laplacian (case $s=1$) introduced in \cite{dTLi21}:
\begin{equation}\label{Mlocal}
\mathcal{M}^p_{r,1}[\phi](x):=\frac{\kappa_{p,d}}{r^p}\fint_{\partial B_r}J_p(\phi(x+y)-\phi(x))\dd \sigma(y),\quad \mathcal{M}^p_{r,2}[\phi](x):=\frac{(p+d)\kappa_{p,d}}{d\,r^p}\fint_{B_r}J_p(\phi(x+y)-\phi(x))\dd y,
\end{equation}
where
$$\kappa_{p,d}:=2\left(\fint_{\partial B_1}|y_1|^p\dd \sigma(y)\right)^{-1}.$$
We present two results. In the first one we prove local uniform convergence of the previous expansions (both local and non local) at every point in the space, with orders of convergence related to the quantity
\begin{equation}\label{def:nu}
\gamma:=\left\{
\begin{split}
2\quad&\mbox{ if } p=2 \mbox{ or }p\geq 4\\
p-2\quad&\mbox{ if }p\in (2,4).
\end{split}\right.
\end{equation}
The result is the following:
\begin{theorem}\label{asympExpUnif}
Let $p \geq 2$ and $s\in(0,1)$. Let $x\in\R^d$ and $0<R<1$ such that $\phi\in C^4_b(B_R(x))$. Then, for $r<R/2$ and any $\gamma$ satisfying \eqref{def:nu}  we have 
\begin{equation}\label{LocalUnif}
\Delta_p\phi=\mathcal{M}_{r,i}^p[\phi]+O(r^{\gamma})\quad \text{uniformly in} \quad B_{R/2}(x) \quad \text{for} \quad i=1,2.
\end{equation}
If additionally $\phi \in L^\infty(\R^d)$, then
\begin{equation}\label{NonLocalUnif}
-(-\Delta)^s_p\phi=\mathcal{M}_r^{s,p}[\phi]+O(r^{\gamma+p(1-s)})\quad \text{uniformly in} \quad B_{R/2}(x).
\end{equation}
\end{theorem}
\begin{remark} \label{rem:optim}
The results in Theorem \ref{asympExpUnif} are optimal. In the range $p\in(2,4)$, take for instance $\phi:\R\to\R$ given by $\phi(x):=\min\{x^2,1\}$ (so that $\nabla \phi(0)=0$). It is standard to get that, for $s\in(0,1)$, we have
\[
|-(-\Delta)^s_p\phi(0)-\mathcal{M}_r^{s,p}[\phi](0)|= C_{s,p} r^{p-2 + p(1-s)} \quad \textup{with} \quad C_{s,p}\not=0 \quad \textup{if} \quad p>2,
\]
and for $j=1,2$,
\[
|\Delta_p \phi(0)-\mathcal{M}_{r,i}^{p}[\phi](0)|= C_{p,j} r^{p-2} \quad \textup{with} \quad C_{p,j}\not=0 \quad \textup{if} \quad p>2.
\]
For $p\geq 4$, we consider $\varphi(x):=\min\{e^{x},2\}$ and compute directly 
\[
|\Delta_p \varphi(0)-\mathcal{M}_{r,i}^{p}[\varphi](0)|= C_{p} r^{2} + O(r^3)\quad \textup{with}\quad  C_p\not=0,
\]
and a Taylor expansion  can be used to see that for $r$ small enough 
\[
|-(-\Delta)^s_p\varphi(0)-\mathcal{M}_r^{s,p}[\varphi](0)|=C_{s,p}r^{2+p(1-s)}+O(r^{3+p(1-s)}) \quad \textup{with}\quad  C_{s,p}\not=0.
\]
\end{remark}
According to the second example above, it makes sense to expect a better convergence in the range $p\in (1,4)$ far from the points where the grandient vanishes. Actually, if we consider
\begin{equation}\label{def:gamma}
\gamma \in \left\{\begin{split}
\{2\}\quad&\mbox{ if }p=2 \mbox{ or }p\geq 3,\\
(1,p-1)\quad&\mbox{ if }p\in (2,3),\\
(0,p-1)\quad&\mbox{ if }p\in (1,2),
\end{split}\right.
\end{equation}
we can prove the second result, that also covers the singular range $p\in (1,2)$. We present here the result for dimensions $d\geq2$, since in dimension $d=1$ we can get better estimates.
\begin{theorem}\label{asympExp}
Let $d\geq2$, $p > 1$ and $s\in(0,1)$. Let $x\in\R^d$ and $0<R<1$ such that $\phi\in C^4_b(B_R(x))$ and $\nabla\phi\neq 0$ in $\overline{B_R(x)}$. Consider 
\begin{equation}\label{def:R0}
R_0:=\frac{\inf_{B_R(x)}{|\nabla\phi|}}{\sqrt{2}\|D^2\phi\|_{L^\infty(B_R(x))}}.
\end{equation}
Then, for $r<\tilde{R}:=\min\{R_0,R/2\}$ and any $\gamma$ satisfying \eqref{def:gamma}  we have 
\begin{equation}\label{LocalGrad}
\Delta_p\phi=\mathcal{M}_{r,i}^p[\phi]+O(r^{\gamma})\quad \text{uniformly in} \quad B_{\tilde{R}/2}(x) \quad \text{for} \quad i=1,2.
\end{equation}
If additionally $\phi \in L^\infty(\R^d)$, then
\begin{equation}\label{NonLocalGrad}
-(-\Delta)^s_p\phi=\mathcal{M}_r^{s,p}[\phi]+O(r^{\gamma+p(1-s)})\quad \text{uniformly in} \quad B_{\tilde{R}/2}(x).
\end{equation}
\end{theorem}
\begin{remark}
The second example in Remark \ref{rem:optim} is valid for every $p\geq 3$. Since in that case $\nabla\varphi(0)\neq 0$, it also proves the optimality of Theorem \ref{asympExp} in the case $p\geq 3$. An interesting open question is what the optimal error is in the case $p\in(1,3)$. Actually, in dimension $d=1$, we are able to prove the optimal orders (see Theorem \ref{thm:1Dresult}). 

\end{remark}

It is important to note that the order of convergence of our asymptotic expansion improves with $p$, reaching super-quadratic orders and more. We also observe that the orders do not degenerate as $s\to1^-$. We refer to Figure \ref{fig:two graphs} for a visual reference in the whole range $s\in(0,1]$ and $p>1$. 

\begin{figure}[h!]
     \centering
     \begin{subfigure}[b]{0.45\textwidth}
         \centering
         \includegraphics[width=0.65\textwidth]{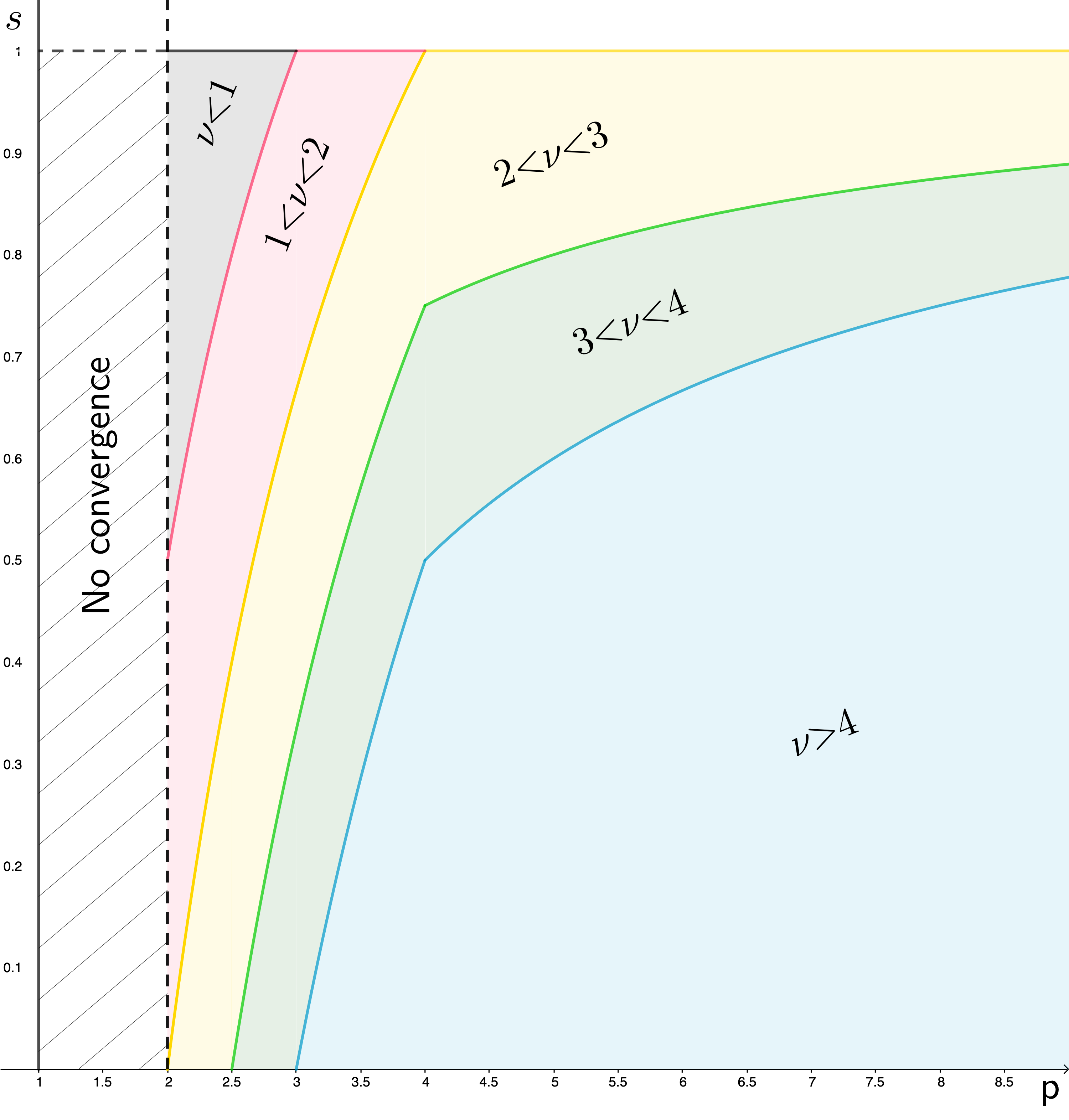}
         \caption{Convergence rate $\nu=\gamma+p(1-s)$ according to Theorem \ref{asympExpUnif}.}
     \end{subfigure}
     \hfill
     \begin{subfigure}[b]{0.45\textwidth}
         \centering
         \includegraphics[width=0.65\textwidth]{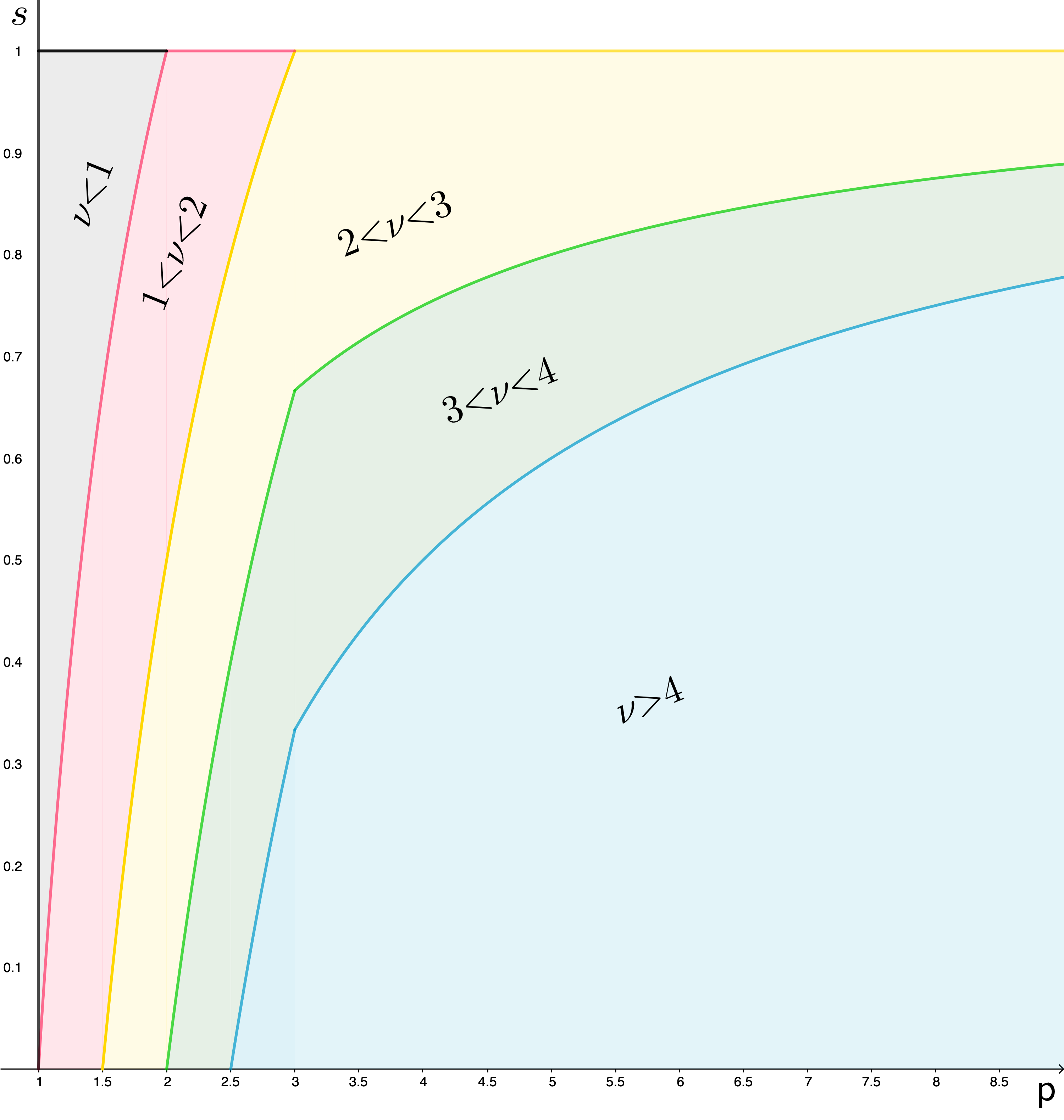}
         \caption{Convergence rate $\nu=\gamma+p(1-s)$ according to Theorem \ref{asympExp}.}
     \end{subfigure}
     \hfill
          \caption{Orders of convergence for the asymptotic expansions for $-(-\Delta)_p^s$ and $\Delta_p$.}
        \label{fig:two graphs}
\end{figure}

Finally, we present the results in dimension $d=1$, where we reach optimal orders of convergence according to Remark \ref{rem:optim}. We would like to note that, in this case, the first asymptotic expansion of the $p$-Laplacian becomes
\begin{equation}\label{eq:asexp1D}
\mathcal{M}^p_{r,1}[\phi](x)= \frac{J_p(\phi(x+r)-\phi(x))+J_p(\phi(x-r)-\phi(x))}{r^p}.
\end{equation}
This is itself a finite difference discretization of the $p$-Laplacian. The consistency result in dimension $d=1$ reads as follows:
\begin{theorem}\label{thm:1Dresult}
Let $d=1$ and the same assumptions and notation of Theorem \ref{asympExp} hold. Then
\begin{equation}\label{LocalGrad}
\Delta_p\phi=\mathcal{M}_{r,i}^p[\phi]+O(r^{2})\quad \text{uniformly in} \quad B_{R/2}(x) \quad \text{for} \quad i=1,2,
\end{equation}
and 
\begin{equation*}
-(-\Delta)^s_p\phi=\mathcal{M}_r^{s,p}[\phi]+O(r^{2+p(1-s)})\quad \text{uniformly in} \quad B_{R/2}(x).
\end{equation*}
\end{theorem}

\subsection{Discretizations}\label{subsec:disc} The second part of the article deals with two discretizations of $\mathcal{M}_r^{s,p}$, which as a consequence will provide discretizations of the fractional $p$-Laplacian. Given $h>0$, consider the uniform grid $\mathcal{G}_h:=h\Z^d=\{y_\alpha^h:=h\alpha:\,\alpha\in\Z^d\}$, and the cubes
 $$Q_\alpha^h:=y_\alpha^h+\frac{h}{2}[-1,1)^d.$$
  For  simplicity we will drop the dependence on $h$ and write $y_\alpha$ and $Q_\alpha$. Given $r>0$ we define, for $i=1,2$, the discrete operators:
\begin{equation}\label{Discret}
-(-\Delta)^{s,h}_{p,i}\phi(x):=\frac{(p+d)}{p(1-s)} \frac{h^d}{ r^{d+sp}}\sum_{y_\alpha\in B_r}J_p(\phi(x+y_\alpha)-\phi(x))+\sum_{y_\alpha\in \R^d\setminus B_r}J_p(\phi(x+y_\alpha)-\phi(x))W_{\alpha,i},
\end{equation}
where 
\begin{equation}\label{def:Wj}
W_{\alpha,1}:=\int_{ Q_\alpha}\frac{\dd y}{|y|^{d+sp}},\qquad W_{\alpha,2}:=\frac{h^d}{|y_\alpha|^{d+sp}}.
\end{equation}

\begin{theorem}\label{DiscretOrder}
Let the hypotheses of either Theorem \ref{asympExpUnif}, Theorem \ref{asympExp} or Theorem \ref{thm:1Dresult} hold, and let $h>0$. Then, there exists $\mu\in(0,1]$ such that, for $r$ appearing in \eqref{Discret} and satisfying $r\asymp h^\mu$, we have
$$(-\Delta)^{s,h}_{p,i}\phi=(-\Delta)^s_p\phi+o_h(1)\quad \mbox{as }h\to 0^+,\quad i=1,2,$$
uniformly in $B_{\tilde{R}/2}(x)$. 
\end{theorem}

Actually, we obtain precise error estimates in terms of $r$ and $h$ in the Theorem above, which in the best cases reach the order $O(h)$. For the sake of simplicity of the presentation, we state them in subsection \ref{subsec:orders}.

The case $p=2$ in Theorem \ref{DiscretOrder} does not need intermediate steps regarding asymptotic expansions, since the discretization of the local Laplacian is already of order $O(h^2)$. This gives a discretization of order $O(h^{2+2(1-s)})$ for the fractional Laplacian (see \cite{HuOb14,dTEnJa18}).

\subsection{Application: explicit finite difference scheme for the fractional Cauchy problem}\label{subs: aplication}
The third part of the paper is devoted to apply Theorem \ref{DiscretOrder} to solve via convergence of finite difference numerical schemes the following parabolic fractional problem
\begin{equation}\label{parabolic problem}
\begin{cases} \partial_t u(x,t)+(-\Delta)_p^{s}u(x, t)=f(x), & \quad (x,t) \in \mathbb{R}^{d}\times(0,T)=:Q_T\\ 
u(x, 0)=u_0(x),   & \quad x \in \mathbb{R}^{d},
\end{cases}
\end{equation}where $T > 0$ and  $f, u_0: \mathbb{R}^{d} \to \mathbb{R}$.  We assume the following conditions on the data
\begin{equation}\label{as:data}\tag{$\textup{A}_{u_0,f}$}
\textup{$u_0$ and $f$ are bounded and globally H\"{o}lder continuous with exponent $a \in (0, 1]$.}
\end{equation}
We will write
$$\Lambda_{u_0}(\delta):=L_{u_0}\delta^{a} \text{ and } \Lambda_{f}(\delta):=L_{f}\delta^{a},$$
where $L_{u_0}:= \max\left\lbrace l_{u_0}, \|u_0\|_{L^{\infty}(\R^d)}\right\rbrace > 0$ and $L_f:=\max \left\lbrace l_f, \|f\|_{L^{\infty}(\R^d)}\right\rbrace>0$, with $l_{u_0}$ and $l_f$ being the H\"{o}lder constants of $u_0$ and $f$, respectively. 

We introduce now the weights $\omega_\alpha=\omega_\alpha(h,r)$ for $\alpha \in \Z^d\setminus\{0\}$ associated to our discretization \eqref{Discret}. More precisely, let
\begin{equation}\label{as:weights}\tag{$\textup{A}_\omega$}
\omega_{\alpha}:=  \left\{ \begin{split}
\frac{(p+d)}{p(1-s)} \frac{h^d}{ r^{d+sp}} \quad &\textup{if} \quad |y_\alpha|<|r|\\
W_{\alpha,i} \quad  &\textup{if} \quad |y_\alpha|\geq|r|, \quad \textup{for either $i=1$ or $i=2$}.
\end{split}\right.
\end{equation}
with $W_{\alpha,i}$ defined in \eqref{def:Wj}.
Finally, we adopt the notation
\begin{equation}\label{eq:Lh}
\LL_h\phi(x):= \sum_{\alpha\not=0} J_p(\phi(x+y_\alpha)-\phi(x)) \omega_\alpha.
\end{equation}
which is just $-(-\Delta)_{p,i}^{s,h}$ given in \eqref{Discret} for either $i=1$ or $i=2$.

We will also introduce a discretization in time whose mesh parameter is $\tau = T/N$, with $N\in \mathbb{N}$. The time grid will be denoted by $\mathcal{T}_\tau:=\{t_j:=\tau j\}_{j=0}^{N}$.

We assume the following stability type hypothesis on discrete parameters:
\begin{equation}\label{as:cfl}\tag{$\textup{A}_{\textup{CFL}}$}
h=O(r), \quad \text{and }\quad 
\tau  \leq  K_{s,p, d}\times \left\{ \begin{split}r^{ 2s+(s-a)(p-2)} \quad &\textup{if} \quad  a<\frac{sp}{p-1}\\
\frac{r^{a}}{|\log(r)|} \qquad \quad \, \, \,  &\textup{if} \quad  a\geq\frac{sp}{p-1}
\end{split}\right.
\end{equation}
where 
\begin{equation}\label{eq:constantCFL}
K_{s,p,d}:=\frac{1}{(p-1)2^{p}C_{s,p, d}\left( L_{u_0}+TL_f +3\tilde{K}_2+1\right)^{p-2}}, 
\end{equation}
with $C_{s,p, d}$ and $\tilde{K}_2$ coming from Lemma \ref{lem:propwe} and Lemma \ref{lem: holder comb}, respectively. 

\begin{remark}
Note that for $a=s$ the CFL condition reads $\tau\lesssim r^{2s}$, recovering the usual stability conditions of the linear nonlocal equation (c.f. \cite{dTEnJa18})
\end{remark}

The  finite difference scheme associated to \eqref{parabolic problem} is defined as:
\begin{equation}\label{scheme}
\begin{cases} U_\alpha^{j}=U_\alpha^{j-1}+\tau \left(\LL_hU_\alpha^{j-1} +f_\alpha\right), & \quad \alpha \in \mathbb{Z}^{d}, j=1, ..., N\\ 
U_\alpha^{0}=(u_0)_\alpha, \qquad  \alpha \in \mathbb{Z}^{d},
\end{cases}
\end{equation}
with $U_\alpha^j=U(x_\alpha,t_j)$ where $U: \mathcal{G}_h\times \mathcal{T}_\tau  \to \R$, and
$(u_0)_\alpha:=u_0(x_\alpha), \, f_\alpha:=f(x_\alpha)$. 
Then, we have the following result.
\begin{theorem}\label{Solutionparabolic prob}
Let $p \in (2, \infty)$ and $s\in(0,1)$. Assume that \eqref{as:weights}, \eqref{as:data} and \eqref{as:cfl} hold and let $U$ be the solution of \eqref{scheme}. Then, for all sequences $h,\tau \to 0^{+}$,  there are   subsequences $h_{k}, \tau_k$ and  a function $u\in C_b(\overline{Q_T})$ such that
for any compact set $K\in Q_T$, we have
\[
\max_{(x_\alpha,t_j)\in K\cap (\mathcal{G}_{h_k}\times \mathcal{T}_{\tau_k})}|U_\alpha^j-u(x_\alpha,t_j)|\to 0 \quad \textup{as} \quad k\to \infty.
\]
Moreover, $u$ is a viscosity solution of problem \eqref{parabolic problem}.
\end{theorem} 
See Definition \ref{def:visc} for more details on the notion of viscosity solutions. Other results for the numerical scheme will be stated and proven in Section \ref{sec:app}. In particular, Theorem \ref{Solutionparabolic prob} shows existence of viscosity solutions of \eqref{parabolic problem}. 

 \section{Asymptotic expansion and orders in the MVP}\label{sec:asymptoticexporder}
 
 The goal of this section is to prove Theorem \ref{asympExpUnif}, Theorem \ref{asympExp} and Theorem \ref{thm:1Dresult}. \normalcolor Let $\phi:\R^d\to \R$ and $y\in \R^d$, $y\not=0$. Let us define, for $p>1$, 
 \begin{equation}\label{defDy}
 D_y[\phi](x) := \frac{J_p(\phi(x+y)-\phi(x))+J_p(\phi(x-y)-\phi(x))}{|y|^p}
 \end{equation}
 where $J_p:\R\to\R$ is given by $J_p(\xi):=|\xi|^{p-2}\xi$. 
The strategy will be the following (avoiding all constants and approximation errors):
\smallskip

\noindent {\bf Step 1.} Use the nonlinear finite difference-type operator \eqref{defDy} to obtain an approximation of the form
$$|y|^pD_y[\phi](x)\approx |\nabla \phi(x)\cdot y|^{p-2}y^TD^2\phi(x) y.$$
{\bf Step 2.} Integrate this expression with respect to the fractional measure to obtain the local $p$-Laplacian
$$\int_{B_r}D_y[\phi](x)\frac{\dd y}{|y|^{d-(1-s)p}}\approx \int_{B_r} |\nabla \phi(x)\cdot y|^{p-2}y^TD^2\phi(x) y\frac{\dd y}{|y|^{d+sp}}\approx \Delta_p\phi(x)r^{p(1-s)}.$$
{\bf Step 3.} In a similar way, 
$$\mathcal{M}^p_{r,j}[\phi](x)\approx\frac{1}{r^p}\fint_{B_r}D_y[\phi](x)\dd y\approx\Delta_p\phi(x).$$
{\bf Step 4.} Combine this information to conclude 
\begin{equation*}\begin{split}
-(-\Delta)^s_p\phi(x)&\approx\frac 12\int_{B_r}D_y[\phi](x)\frac{\dd y}{|y|^{d-(1-s)p}}+\int_{\R^d\setminus B_r}J_p(\phi(x+y)-\phi(x))\frac{\dd y}{|y|^{d+sp}}\\
&\approx \mathcal{M}^p_{r,2}[\phi](x)r^{p(1-s)}+\int_{\R^d\setminus B_r}J_p(\phi(x+y)-\phi(x))\frac{\dd y}{|y|^{d+sp}}\approx\mathcal{M}^{s,p}_{r}[\phi](x).
\end{split}\end{equation*}
This leads to the asymptotic expansion and results for the fractional $p$-Laplacian given in Theorem \ref{asympExpUnif}, Theorem \ref{asympExp} and Theorem \ref{thm:1Dresult}.

Depending on the range of $p$, these steps may be more or less involved. In the case $p\geq 4$, Step 1 is uniformly obtained and the integration is done straightforward, while in the other cases more refined arguments are required to deal with the points where the gradient vanishes. To prove Steps 1-3 we devote a subsection for every range of $p$ and we complete Step 4 at the end of the section. The case $d=1$ is treated separately. 

We start with two auxiliary results that will allow us to perform Step 2 and Step 3.

\begin{lemma}\label{lem:intestim}
Let $\alpha>-1$, $d\geq2$ and $\nabla\phi(x)\not=0$. Consider $\phi\in C^1_b(B_R(x))$ for some $x\in \R^d$ and $0<r< R<1$. Then
\begin{equation}\label{grad phi ineq}
P.V. \fint_{B_r}|\nabla \phi(x)\cdot \frac{y}{|y|}|^\alpha dy\leq \gamma_{d,\alpha}|\nabla\phi(x)|^{\alpha},
\end{equation}
and, for $p\in(1,\infty)$ and $s\in(0,1)$, we have
\[
P.V. \int_{B_r}|\nabla \phi(x)\cdot \frac{y}{|y|}|^\alpha \frac{dy}{|y|^{d-(1-s)p}}\leq \gamma_{s,p,d,\alpha}|\nabla\phi(x)|^{\alpha} r^{(1-s)p}.
\]
\end{lemma}
\begin{proof}
Assume that $\nabla \phi(x)=ce_1$ with $c\not=0$, and thus, $|\nabla\phi(x)|=c$. Changing to polar coordinates:
\begin{equation*}
\begin{split}
I_1:=&\fint_{B_r}|\nabla \phi(x)\cdot \frac{y}{|y|}|^\alpha dy
= \frac{|c|^\alpha}{\omega_d r^d}\int_{B_r}\frac{|y_1|^\alpha}{|y|^\alpha}dy
 =  \frac{|c|^\alpha}{\omega_d r^d} \int_0^{2\pi} \int_0^r\frac{|\rho\cos(\theta)|^\alpha}{\rho^\alpha}\rho^{d-1} \pi^{d-2}\dd \rho \dd \theta\\
=& \frac{2\pi^{d-2} |c|^\alpha}{d\omega_d} \int_0^\pi |\cos(\theta)|^{\alpha} \dd\theta
\end{split}
\end{equation*}
Changing now $\tau=\cos(\theta)$, i.e., $\dd\theta=-\dd \tau/\sqrt{1-\tau^2}$, we get
\[
\begin{split}
I_1&\leq \frac{2\pi^{d-2} |c|^\alpha}{d\omega_d} \int_0^1 \frac{\tau^\alpha}{\sqrt{1-\tau^2}}\dd\tau\leq \gamma_{d,\alpha} |c|^\alpha 
\end{split}
\]
In a similar way,
\begin{equation*}
\begin{split}
I_2:=&\int_{B_r}|\nabla \phi(x)\cdot \frac{y}{|y|}|^\alpha \frac{dy}{|y|^{d-(1-s)p}}
=|c|^\alpha\int_{B_r}\frac{|y_1|^\alpha}{|y|^{\alpha+d-(1-s)p}}dy\\
= &|c|^\alpha \int_0^{2\pi} \int_0^r\frac{|\rho\cos(\theta)|^\alpha}{\rho^{\alpha+d-(1-s)p}}\rho^{d-1} \pi^{d-2}\dd \rho \dd \theta
\leq \gamma_{s,p,d,\alpha} |c|^\alpha 
 r^{(1-s)p}.
\end{split}
\end{equation*}
\end{proof}

\begin{lemma}\label{lem:niceidentity}
Let $p\in(1,\infty)$, $d \geq 1$, and $s\in(0,1)$. Consider $\phi\in C^2_b(B_R(x))$ for some $x\in \R^d$ and $0<R<1$ and such that $\nabla \phi(x)\not=0$ when $p\in(1,2)$. Then,
\begin{equation}\label{J1}
a_{s,p,d}\,P.V.\int_{B_r}(p-1)|\nabla \phi(x)\cdot y|^{p-2}y^TD^2\phi(x) y \frac{\dd y}{|y|^{d+sp}}= \Delta_p\phi(x) r^{p(1-s)}. 
\end{equation}
with $a_{s,p,d}:=(\frac{1}{p(1-s)}  \int_{\partial B_1} |y_1|^{p} \dd \sigma(y))^{-1}.$ We also have that
\begin{equation}\label{J2}
\begin{split}
a_{p,d}\fint_{\partial B_r} (p-1)|\nabla \phi(x)\cdot y|^{p-2}y^TD^2\phi(x) y \dd \sigma(y)& =\Delta_p \phi(x)r^p,\\ 
\frac{(p+d)a_{p,d}}{d}\fint_{B_r} (p-1)|\nabla \phi(x)\cdot y|^{p-2}y^TD^2\phi(x) y \dd y & = \Delta_p \phi(x)r^p,
\end{split}
\end{equation}
with $a_{p,d}:= \left(\fint_{\partial B_1} |y_1|^{p} \dd \sigma(y) \right)^{-1}$.
\end{lemma}
\begin{proof}Let $d\geq2$ for now. If $p>2$ and $\nabla\phi(x)=0$ the identity trivially holds since both sides in \eqref{J1} are zero, so we consider the case $\nabla \phi(x)\neq 0$. Assume without loss of generality that $\nabla \phi(x)=ce_1$ for some $c>0$. Then, for $0 <\varepsilon < r$,
\begin{equation*}\begin{split}
\int_{B_r\setminus B_\varepsilon} (p-1)&|\nabla \phi(x)\cdot y|^{p-2}y^TD^2\phi(x) y \frac{\dd y}{|y|^{d+sp}}\\
&=c^{p-2} \int_\varepsilon^r \rho^{-d-sp}(p-1)\int_{\partial B_\rho}|y_1|^{p-2}  \left(\sum_{i=1}^d y_i^2 D_{ii}\phi(x)\right)\dd \sigma(y) \dd \rho. 
\end{split}\end{equation*}
Furthermore,
\[
(p-1) \int_{\partial B_\rho} |y_1|^{p-2}y_i^2 \dd \sigma(y)= \int_{\partial B_\rho} |y_1|^{p} \dd \sigma(y) = \rho^{d+p-1} \int_{\partial B_1} |y_1|^{p} \dd \sigma(y) \quad \textup{for all} \quad i=2,\ldots,d,
\]
and
\[
\begin{split}
\int_{\partial B_\rho}|y_1|^{p-2}  \left(\sum_{i=1}^d y_i^2 D_{ii}\phi(x)\right)\dd \sigma(y)&= D_{11}\phi(x) \int_{\partial B_\rho} |y_1|^{p} \sigma(y) + \sum_{i=2}^d D_{ii} \phi(x) \int_{\partial B_\rho} |y_1|^{p-2}y_i^2 \dd \sigma(y)\\
&=\left(D_{11}\phi(x)  + \frac{1}{p-1} \sum_{i=2}^d D_{ii} \phi(x)  \right) \rho^{d+p-1} \int_{\partial B_1} |y_1|^{p} \dd \sigma(y)\\
&=\frac{\rho^{d+p-1} }{p-1} \left(\sum_{i=1}^d D_{ii} \phi(x)  + (p-2) D_{11}\phi(x) \right) \int_{\partial B_1} |y_1|^{p} \dd \sigma(y).
\end{split}
\]
Finally, we note that
\[
c^{p-2}\left( \sum_{i=1}^d D_{ii} \phi(x)  + (p-2) D_{11}\phi(x)\right)= \Delta_p\phi(x),
\]
to get
\[
\begin{split}
\int_{B_r\setminus B_\varepsilon} (p-1)|\nabla \phi(x)\cdot y|^{p-2}y^TD^2\phi(x) y \frac{\dd y}{|y|^{d+sp}}&= \Delta_p\phi(x)  \int_{\partial B_1} |y_1|^{p} \dd \sigma(y) \int_{\varepsilon}^r \rho^{-1+p(1-s)}\dd \rho\\
&= a_{d,p,s}^{-1} \Delta_{p}\phi(x) r^{p(1-s)}+ o_{\varepsilon}(1).
\end{split}
\]
As $\varepsilon \to 0$, we obtain \eqref{J1}.

Proceeding analogously,
\[
\begin{split}
\fint_{B_r} (p-1)|\nabla \phi(x)\cdot y|^{p-2}y^TD^2\phi(x) y \dd y&= \frac{c^{p-2}}{\omega_d r^d} \int_0^r (p-1)\int_{\partial B_\rho}|y_1|^{p-2}  \left(\sum_{i=1}^d y_i^2 D_{ii}\phi(x)\right)\dd \sigma(y) \dd \rho\\
&=\frac{1}{\omega_d r^d} \Delta_p\phi(x) \int_{\partial B_1} |y_1|^{p} \dd \sigma(y) \int_0^r \rho^{d+p-1}\dd\rho\\
&= r^{p} \frac{d}{p+d}a_{p,d}^{-1} \Delta_p \phi(x).
\end{split}
\]
The above computations also hold in dimension $d=1$ in a simpler way.
\end{proof}

 \subsection{Case $p\geq 4$.} To prove Step 1 we start by doing a Taylor expansion up to order three of $J_p$, which in this case can be done avoiding singular terms.
 
 \begin{lemma}\label{Approx p greater 4}
Let $p\geq 4$. Consider a function $\phi\in C^4_b(B_R(x))$ for some $x\in \R^d$ and $0<R<1$. Then, for any $y\in B_{\frac{R}{2}}(0)\setminus\{0\}$ we have
\[
|D_y[\phi](x)- \frac{(p-1)}{|y|^p}|\nabla \phi(x)\cdot y|^{p-2}y^TD^2\phi(x) y|\leq C(p,d,\|\phi\|_{C^4_b(B_R(x))})|y|^2.
\]
 \end{lemma}
 
 \begin{proof}
By doing a fourth-order Taylor expansion we can write
\begin{equation*}
\begin{split}
\phi(x\pm y)-\phi(x)&=\underbrace{\pm\nabla \phi(x)\cdot y+\frac{1}{2}y^{T}D^{2}\phi(x)y \pm \frac{1}{3!}\sum_{|\alpha|=3}D^{\alpha}\phi(x)y^{\alpha}}_{\phi_{13}^{\pm}(y)}+O(|y|^{4})
\end{split}
\end{equation*}
and thus
\begin{equation}\label{Taylor11}
\begin{split}
|y|^{p}D_y[\phi](x)&=J_p( \phi_{13}^+(y)+O(|y|^{4})) +J_p( \phi_{13}^-(y)+O(|y|^{4})).
\end{split}
\end{equation}
Furthermore, since $|\phi_{13}^{\pm}(y)|=O(|y|)$,
\begin{equation}\label{Taylor12}
\begin{split}
&  J_p( \phi_{13}^\pm(y)+O(|y|^{4})) =  J_p( \phi_{13}^\pm(y))+J_p'(O(|y|))O(|y|^{4})= J_p( \phi_{13}^\pm(y))+O(|y|^{p+2}),
\end{split}
\end{equation}Hence, plugging \eqref{Taylor12} into \eqref{Taylor11} yields
\begin{equation}\label{Taylor14}
\begin{split}
|y|^{p}D_y[\phi](x)&= J_p( \phi_{13}^+(y)) + J_p( \phi_{13}^-(y))+O(|y|^{p+2}).
\end{split}
\end{equation}
Let us denote
$$\phi_{12}^{\pm}(y):=\pm\nabla \phi(x)\cdot y+\frac{1}{2}y^{T}D^{2}\phi(x)y.$$
Performing a second-order Taylor expansion we can write
\begin{equation*}
\begin{split}
 J_p( \phi_{13}^{\pm}(y)) &=   J_p( \phi_{12}^{\pm}(y)) +J_p'(\phi_{12}^{\pm}(y))\bigg(\pm\frac{1}{3!}\sum_{|\alpha|=3}D^{\alpha}\phi(x)y^{\alpha}\bigg)+\frac{1}{2}J_p''\bigg( O(|y|)\bigg)O(|y|^{6}) \\ & =J_p( \phi_{12}^{\pm}(y))+  J_p'(\phi_{12}^{\pm}(y))\bigg(\pm\frac{1}{3!}\sum_{|\alpha|=3}D^{\alpha}\phi(x)y^{\alpha}\bigg) +O(|y|^{p+3}).
\end{split}
\end{equation*}
Using this in \eqref{Taylor14} we obtain 
\begin{equation}\label{Taylor17}
\begin{split}
|y|^{p}D_y[\phi](x)=J_p( \phi_{12}^+(y))+J_p(\phi_{12}^-(y))+\frac{1}{3!}\sum_{|\alpha|=3}D^{\alpha}\phi(x)y^{\alpha} \bigg[ J_p'(\phi_{12}^+(y)) - J_p'(\phi_{12}^-(y))\bigg]+O(|y|^{p+2}).
\end{split}
\end{equation}
Next, we do a first-order Taylor expansion for $J_p'$ to get
\begin{equation}\label{Taylor18}
\begin{split}
 J_p'(\phi_{12}^+(y)) - J_p'(\phi_{12}^-(y))=  J_p'(\nabla \phi(x)\cdot y) - J_p'(-\nabla \phi(x)\cdot y) +O(|y|^{p-1}) = O(|y|^{p-1}).
\end{split}
\end{equation}
Moreover, since $p\geq 4$,
\begin{equation}\label{Taylor19}
\begin{split}
 J_p( \phi_{12}^+(y))&+J_p(\phi_{12}^-(y)) = J_p(\nabla\phi(x)\cdot y)+ \frac{1}{2}J'_p(\nabla \phi(x) \cdot y)y^{T}D^{2}\phi(x)y + \frac{1}{8}J_p''(\nabla \phi(x)\cdot y)(y^{T}D^{2}\phi(x)y)^{2} \\ 
 & \quad + J_p(-\nabla\phi(x)\cdot y)+ \frac{1}{2}J'_p(-\nabla \phi \cdot y)y^{T}D^{2}\phi(x)y + \frac{1}{8}J_p''(-\nabla \phi(x)\cdot y)(y^{T}D^{2}\phi(x)y)^{2} + O(|y|^{p+2}) \\ & =J'_p(\nabla \phi \cdot y)y^{T}D^{2}\phi(x)y + O(|y|^{p+2}).
\end{split}
\end{equation}
By \eqref{Taylor17}, \eqref{Taylor18} and \eqref{Taylor19}, we obtain
\begin{equation*}
\begin{split}
|y|^{p}D_y[\phi](x)&=(p-1)|\nabla \phi(x)\cdot y|^{p-2}y^{T}D^{2}\phi(x)y+O(|y|^{p+2}).
\end{split}
\end{equation*}This ends the proof of the Lemma. 
\end{proof}
Notice that in the previous proof the fact that $p\geq 4$ has been crucial to estimate the reminder terms in the Taylor expansions. 
 
\begin{theorem}\label{teo4}
Let $p \geq 4$ and $s\in(0,1)$. Consider a function $\phi\in C^4_b(B_R(x_0))$ for some $x_0\in \R^d$ and $0<R<1$. Then, for $r<R/2$ and $x\in B_{R/2}(x_0)$ we have that
\begin{equation}\label{Id1}
\,P.V.\int_{B_r}J_p(\phi(x+y)-\phi(x)) \frac{\dd y}{|y|^{d+sp}}= \frac{1}{2a_{s,p,d}}\Delta_p\phi(x) r^{p(1-s)} + O(r^{2+p(1-s)}),
\end{equation}
with $a_{s,p,d}$ given in Lemma \ref{lem:niceidentity}, and
\begin{equation}\label{Id2}
\mathcal{M}_{r,i}^p[\phi](x)= \Delta_p\phi(x) + O(r^{2}),\quad i\in\{1,2\},
\end{equation}
where the error terms are uniform in $x$ and $\mathcal{M}_{r,i}^p$ is given in  \eqref{Mlocal}.
\end{theorem}

\begin{proof}
Observe that, by \eqref{defDy} and Lemma \ref{Approx p greater 4},
\[
\begin{split}
P.V.& \int_{B_r}J_p(\phi(x+y)-\phi(x)) \frac{\dd y}{|y|^{d+sp}}=\frac{1}{2}\,P.V. \int_{B_r} D_{y}[\phi](x)\frac{\dd y }{|y|^{d-(1-s)p}}\\
&=\frac{1}{2}(p-1)\int_{B_r}|\nabla \phi(x)\cdot y|^{p-2}y^TD^2\phi(x) y\frac{\dd y }{|y|^{d+sp}}+O\left(\int_{B_r}|y|^{2}\frac{\dd y }{|y|^{d-(1-s)p}}\right)\\
&= \frac{1}{2}(p-1)\int_{B_r}|\nabla \phi(x)\cdot y|^{p-2}y^TD^2\phi(x) y\frac{\dd y }{|y|^{d+sp}}+O(r^{2+(1-s)p}).
\end{split}
\]
Then \eqref{Id1} follows by Lemma \ref{lem:niceidentity}. Identity \eqref{Id2} can be analogously obtained integrating in $B_r$ with respect to the measure $\frac{\dd y}{r^{p}|B_r|}$ or integrating on $\partial B_r$ with respect to $\frac{\dd \sigma(y)}{r^p|\partial B_r|}$.
\end{proof}
 
 \subsection{Case $p\in [3,4)$.} We start by doing a Taylor expansion as in the range $p\geq 4$. However, this argument only works up to $J_p''$, since the power $p-4$ is negative and singular terms may appear. This provides a uniform error, which in this case is not quadratic, as it is shown in the next Lemma. In order to obtain a quadratic error we need to be away of zero gradient points, and the argument becomes more involved (see Lemma \ref{Approx p greater 3} and Theorem \ref{theo_p34}).
 
 \begin{lemma}\label{p greater 3}
Let $p\in [3,4)$. Consider a function $\phi\in C^4_b(B_R(x))$ for some $x\in \R^d$ and $0<R<1$. Then, for any $y\in B_{\frac{R}{2}}(0)\setminus\{0\}$ we have
\[
|D_y[\phi](x)- \frac{(p-1)}{|y|^p}|\nabla \phi(x)\cdot y|^{p-2}y^TD^2\phi(x) y|\leq C(p, d,\|\phi\|_{C^4_b(B_R(x))})|y|^{p-2}.
\] 
\end{lemma}
\begin{proof}
Let us define
\begin{equation}\label{phi23}
\phi_{23}^{\pm}(y):=\frac{1}{2}y^{T}D^{2}\phi(x)y \pm \frac{1}{3!}\sum_{|\alpha|=3}D^{\alpha}\phi(x)y^{\alpha},
\end{equation}
and, given $\theta_j\in (0,1),$
\begin{equation}\label{thetaj}
\Theta_j^{\pm}(y):=\theta_j(\pm\nabla \phi(x)\cdot y) + (1-\theta_j)\bigg(\pm\nabla \phi(x)\cdot y +\phi_{23}^{\pm}(y)+O(|y|^{4})\bigg).
\end{equation}
We perform a fourth-order Taylor expansion of $\phi$ around $x$ and then a second-order expansion for $J_p$ at $\nabla \phi(x)\cdot y$, to get
\begin{equation*}
\begin{split}
|y|^{p}D_y[\phi](x)=&\,J_p(\nabla \phi(x)\cdot y)+J_p(-\nabla \phi(x)\cdot y)  \\
&+ J_p'(\nabla \phi(x)\cdot y)\cdot ( \phi_{23}^+(y)+O(|y|^{4})) +J_p'(-\nabla \phi(x)\cdot y)\cdot ( \phi_{23}^-(y)+O(|y|^{4}))\\
&+\frac{1}{2}J_p''(\Theta_1^+(y))(\phi_{23}^+(y)+O(|y|^{4}))^{2}+\frac{1}{2}J_p''(\Theta_2^-(y))(\phi_{23}^-(y)+O(|y|^{4}))^{2},
\end{split}
\end{equation*}
for some $\theta_1, \theta_2 \in (0,1)$. Recalling that $J_p$ and $J_p'$ are odd and even functions respectively, we get
\begin{equation}\label{Taylor7}
\begin{split}
|y|^{p}D_y[\phi](x)=&\,(p-1)|\nabla \phi(x)\cdot y|^{p-2}y^{T}D^{2}\phi(x)y + O(|y|^{p+2}) \\ & +\frac{1}{2}J_p''(\Theta_1^+(y))(\phi_{23}^+(y)+O(|y|^{4}))^{2}+\frac{1}{2}J_p''(\Theta_2^-(y))(\phi_{23}^-(y)+O(|y|^{4}))^{2}.
\end{split}
\end{equation}
Adding and substracting in \eqref{Taylor7} the terms
$$\frac{1}{2}J_p''(\pm\nabla \phi(x)\cdot y)\cdot(\phi_{23}^\pm(y)+O(|y|^{4} ))^{2},$$
we get
\begin{equation}\label{Taylor8}
\begin{split}
|y|^{p}D_y[\phi](x)&= (p-1)|\nabla \phi(x)\cdot y|^{p-2}y^{T}D^{2}\phi(x)y + O(|y|^{p+2}) \\
& + \frac{1}{2}\bigg[ J_p''(\Theta_1^+(y))-J_p''(\nabla \phi(x)\cdot y)\bigg](\phi_{23}^+(y)+O(|y|^{4} ))^{2} +\frac{1}{2}J_p''(\nabla \phi(x)\cdot y)(\phi_{23}^+(y)+O(|y|^{4} ))^{2} \\ 
& + \frac{1}{2}\bigg[ J_p''(\Theta_2^-(y))-J_p''(\nabla \phi(x)\cdot y)\bigg](\phi_{23}^-(y)+O(|y|^{4} ))^{2} +\frac{1}{2}J_p''(-\nabla \phi(x)\cdot y)(\phi_{23}^-(y)+O(|y|^{4} ))^{2}.
\end{split}
\end{equation}
By \cite[Lemma 2.4]{BaMe21} applied to $J_p''$ for $p \in (3,4)$,  we get
\begin{equation*}
\begin{split}
&\bigg[ J_p''(\Theta_1^+(y)) -J_p''(\nabla \phi(x)\cdot y)\bigg] (\phi_{23}^+(y)+O(|y|^{4} ))^{2}= O(|y|^{2(p-3)})\cdot O(|y|^{4})=O(|y|^{2p-2}),
\end{split}
\end{equation*}and similarly for the term with $\Theta_2^-$. Finally, recalling that $J_p''$ is odd,  the remianing terms in \eqref{Taylor8} are of order $O(|y|^{p+3})$. Then the result is achieved putting all the estimates together.
\end{proof}
Notice that the order obtained in this Lemma is $p-2$, which is not quadratic when $p\in[3,4)$. Nevertheless, we can refine the argument in a certain region (see the following Lemma) to obtain quadratic error, which is large enough so that the integral in the complementary part preserves the quadratic order (see Theorem \ref{theo_p34}).
 
 \begin{lemma}\label{Approx p greater 3}
Let $p\in [3,4)$. Consider a function $\phi\in C^4_b(B_R(x))$ for some $x\in \R^d$ and $0<R<1$, and let $K_0:=\|D^2\phi\|_{L^\infty(B_R(x))}$. Then, for any $y\in B_{\frac{R}{2}}(0)\setminus\{0\}$ in the region $|\nabla\phi(x)\cdot y|> K_0|y|^2$, we have
\[
|D_y[\phi](x)- \frac{(p-1)}{|y|^p}|\nabla \phi(x)\cdot y|^{p-2}y^TD^2\phi(x) y|\leq C(p,d,\|\phi\|_{C_b^4(B_R(x))})\Phi\left(\nabla \phi(x)\cdot \frac{y}{|y|}\right)|y|^2,
\]
where  $
\Phi(\eta):=1+(p-3)|\eta|^{p-4}$.
 \end{lemma}
 
  \begin{proof}
 Let $f,g:\R^d\to\R$ be smooth functions such that $f(y)\not=0$ for $y\not=0$ and $|g(y)/f(y)|<\frac 12$. Then
\begin{equation}\label{eq:estMaria}
 \begin{split}
 J_p&(f(y)+g(y))=J_p(f(y))J_p\left(1+\frac{g(y)}{f(y)}\right)\\
 &=J_p(f(y))\left(J_p(1) +  J_p'(1) \frac{g(y)}{f(y)}+ \frac{J_p''(1)}{2} \left(\frac{g(y)}{f(y)}\right)^2+ \frac{J_p'''(\eta)}{3!} \left(\frac{g(y)}{f(y)}\right)^3\right)\\
 &=J_p(f(y))+(p-1)|f(y)|^{p-2}g(y)+\frac{(p-1)(p-2)}{2}|f(y)|^{p-2}\frac{g(y)^2}{f(y)}+\frac{J_p'''(\eta)}{3!} |f(y)|^{p-2} \frac{g(y)^3}{f(y)^2}
 \end{split}
\end{equation}
 where $\eta \in B_{\frac{g(y)}{f(y)}}(1)\subset B_{\frac{1}{2}}(1)$, that is, $1/2\leq |\eta|\leq 1$. Now we consider
 \begin{equation} \label{fgdef}
 \phi_1(y):=\nabla \phi(x)\cdot y \quad \textup{and} \quad \phi_{24}^{\pm}(y):=\frac{1}{2}y^{T}D^2\phi(x) y \pm \frac{1}{3!} \sum_{|\alpha|=3}D^\alpha\phi(x) y^{\alpha} + O(|y|^4).
 \end{equation}
 By a direct application of the Mean Value Theorem on $\phi(x\pm y)-\phi(x)$ we get the bound
 $$|\phi_{24}^{\pm}(y)|\leq \frac 12\|D^2\phi\|_{L^\infty (B_R(x))}|y|^2,$$
and hence, if $y$ belongs to the region $|\nabla\phi(x)\cdot y|> K_0|y|^2$ then
 $$\bigg|\dfrac{\phi_{24}^{\pm}(y)}{\phi_1(y)}\bigg|<\frac{1}{2}.$$
We use now the definition of $D_y[\phi]$ together with \eqref{eq:estMaria} to obtain
 \[
 \begin{split}
 |y|^{p}&D_y[\phi](x)=J_p(\phi_1(y)+\phi_{24}^+(y))+J_p(-\phi_1(y)+\phi_{24}^{-}(y))\\
 =&\underbrace{J_p(\phi_1(y))+J_p(-\phi_1(y))}_{I_1}+\underbrace{(p-1)|\phi_1(y)|^{p-2}(\phi_{24}^{+}(y)+\phi_{24}^{-}(y))}_{I_2}\\
 &+\underbrace{\frac{(p-1)(p-2)}{2}|\phi_1(y)|^{p-2}\left(\frac{\phi_{24}^+(y)^2}{\phi_1(y)}-\frac{\phi_{24}^-(y)^2}{\phi_1(y)}\right)}_{I_3}+\underbrace{ \frac{1}{3!}|\phi_1(y)|^{p-2}\left(J_p'''(\eta_1) \frac{\phi_{24}^+(y)^3}{\phi_1(y)^2}+J_p'''(\eta_2) \frac{\phi_{24}^-(y)^3}{\phi_1(y)^2}\right)}_{I_4}.
 \end{split}
 \]
By the antisymmetry of $J_p$ we have $I_1=0$. On the other hand,
\begin{equation*}
\begin{split}
|I_2- (p-1)|\nabla \phi(x)\cdot y|^{p-2}y^TD^2\phi(x) y|&\leq C|\nabla \phi(x) \cdot y|^{p-2}|y|^4\leq C(p,\|\phi\|_{C_b^4(B_R(x))})|y|^{p+2}.
\end{split}
\end{equation*}
Now we estimate $I_3$. First, we note that
\[
\begin{split}
|\phi_{24}^+(y)^2-\phi_{24}^{-}(y)^2|&= |(\phi_{24}^+(y)+\phi_{24}^{-}(y))(\phi_{24}^+(y)-\phi_{24}^{-}(y))|
\leq (|y^TD^2\phi(x)y|+O(|y|^4)) (O(|y|^3)+O(|y|^4))\\
&\leq C|y|^5,
\end{split}
\]
where we have used that $|y|\leq1/2$. Then
\begin{equation*}\label{eq:seconderrorterm}
\begin{split}
|I_3|\leq& \frac{(p-1)(p-2)}{2} |\nabla \phi(x)\cdot \frac{y}{|y|}|^{p-3} |y|^{p-3}|\phi_{24}^+(y)^2-\phi_{24}^{-}(y)^2|
\leq  C(p,\|\phi\|_{C_b^4(B_R(x))}) |y|^{p+2}.
\end{split}
\end{equation*}
Finally, 
\[
\begin{split}
|I_4|\leq&\,C(p-3)|y|^{p-4} |\nabla \phi(x)\cdot \frac{y}{|y|}|^{p-4} 
||y|^2+|y|^3 
+ |y|^4|^3\leq C(p-3)|\nabla \phi(x)\cdot \frac{y}{|y|}|^{p-4} |y|^{p+2}.
\end{split}
\]\end{proof} 
Combing the results in Lemma \ref{p greater 3} and Lemma \ref{Approx p greater 3} we can obtain quadratic order in the error away from the zero gradient points.
  
  \begin{theorem}\label{theo_p34}
Let $p \in [3,4)$ and $s\in(0,1)$. Consider a function $\phi\in C^4_b(B_R(x_0))$ for some $x_0\in \R^d$ and $0<R<1$, and such that $\nabla\phi\not=0$ in $\overline{ B_R(x_0)}$. Let $R_0$ defined in \eqref{def:R0}. Then, for $r<\min\{R_0,R/2\}$ and $x\in B_{R/2}(x_0)$ we have that
\begin{equation}\label{Id134}
\,P.V.\int_{B_r}J_p(\phi(x+y)-\phi(x)) \frac{\dd y}{|y|^{d+sp}}= \frac{1}{2a_{s,p,d}}\Delta_p\phi(x) r^{p(1-s)} + O(r^{2+p(1-s)})
\end{equation}
with $a_{s,p,d}$ given in Lemma \ref{lem:niceidentity} and
\begin{equation}\label{Id234}
\mathcal{M}_{r,i}^p[\phi](x)= \Delta_p\phi(x) + O(r^{2}),\quad i\in\{1,2\},
\end{equation}
where the error terms are uniform in $x$ and $\mathcal{M}_{r,i}^p$ is given in \eqref{Mlocal}.
\end{theorem}

\begin{proof}
Let $K_0:= \|D^2\phi\|_{L^\infty(B_R(x))}$. Consider the region
$$\Omega:=\left\{y\in\R^d:\,|\nabla\phi(x)\cdot y|>K_0|y|^2\right\},$$
and define
\begin{equation}\begin{split}\label{I1I2}
I_1:=&\,\,P.V.\int_{B_r\cap \Omega}J_p(\phi(x+y)-\phi(x)) \frac{\dd y}{|y|^{d+sp}}=\frac{1}{2}\,P.V.\int_{B_r\cap \Omega}D_y[\phi](x) \frac{\dd y}{|y|^{d-p(1-s)}},\\
I_2:=&\,\,P.V.\int_{B_r\cap (\R^d\setminus\Omega)}J_p(\phi(x+y)-\phi(x)) \frac{\dd y}{|y|^{d+sp}}=\frac{1}{2}\,P.V.\int_{B_r\cap(\R^d\setminus\Omega)}D_y[\phi](x) \frac{\dd y}{|y|^{d-p(1-s)}}.
\end{split}\end{equation}
By Lemma \ref{Approx p greater 3} and Lemma \ref{lem:intestim} we have that
\begin{equation}
\begin{split}\label{I1_est}
I_1=&\,\frac{1}{2}(p-1)\int_{B_r \cap \Omega} 
|\nabla \phi (x)\cdot y|^{p-2}y^{T}D^{2}\phi(x)y \,\frac{\dd y }{|y|^{d+sp}}  \\
&+r^{2}O\left(\int_{B_r \cap \Omega}  \left(|\nabla \phi(x)\cdot \frac{y}{|y|}|^{p-2}+|\nabla \phi(x)\cdot \frac{y}{|y|}|^{p-3}+(p-3)|\nabla \phi(x)\cdot \frac{y}{|y|}|^{p-4}\right)\frac{\dd y }{|y|^{d-(1-s)p}}  \right)
\\= & \,\frac{1}{2}(p-1)\int_{B_r \cap \Omega} 
|\nabla \phi (x)\cdot y|^{p-2}y^{T}D^{2}\phi(x)y \,\frac{\dd y }{|y|^{d+sp}}  +O(r^{2+(1-s)p}).
\end{split}
\end{equation}
By Lemma \ref{p greater 3}, we can estimate $I_2$ as
\begin{equation}\begin{split}\label{I2_est}
I_2=&\,\frac{1}{2}(p-1)\int_{B_r \cap (\R^d\setminus\Omega)} 
|\nabla \phi (x)\cdot y|^{p-2}y^{T}D^{2}\phi(x)y \,\frac{\dd y }{|y|^{d+sp}}+O\left( r^{p-2}\int_{B_r \cap (\R^d\setminus\Omega)} \frac{\dd y }{|y|^{d-(1-s)p}} \right).
\end{split}\end{equation}
Without loss of generality we can assume $\nabla \phi(x)=|\nabla\phi(x)|e_1$. Thus, asking $y\in \R^d\setminus\Omega$ is equivalent to require $|y_1|\leq\frac{K_0}{|\nabla\phi(x)|}|y|^2$, where $y_1$ denotes the first coordinate of $y$. Using standard polar coordinates, this can be rewritten as $|\cos\theta|<\frac{K}{|\nabla\phi(x)|}r$ and therefore
\begin{equation*}\begin{split}
\int_{B_r \cap (\R^d\setminus\Omega)}  \frac{\dd y }{|y|^{d-(1-s)p}}&=\pi^{d-2}\int_0^r\int_{\arccos(\frac{K}{|\nabla\phi(x)|}\rho)}^{\pi-\arccos(\frac{K}{|\nabla\phi(x)|}\rho)}\frac{\dd \theta \dd\rho }{\rho^{1-(1-s)p}}\\
&=\pi^{d-2}\int_0^r\left(\pi-2\arccos\left(\frac{K_0}{|\nabla\phi(x)|}\rho\right)\right)\frac{ \dd\rho }{\rho^{1-(1-s)p}}.
\end{split}\end{equation*} 
Choosing $r<R_0$ it can be seen that 
$$\bigg|\pi-2\arccos\left(\frac{K_0}{|\nabla\phi(x)|}\rho\right)\bigg|\leq 4 \frac{K_0}{|\nabla\phi(x)|}\rho,$$
and hence 
$$\int_{B_r \cap (\R^d\setminus\Omega)}  \frac{\dd y }{|y|^{d-(1-s)p}}=O\left(\int_0^r\rho^{(1-s)p}\dd\rho\right)=O(r^{1+(1-s)p}).$$
Replacing in \eqref{I2_est} we get
$$I_2=\,\frac{1}{2}(p-1)\int_{B_r \cap (\R^d\setminus\Omega)} 
|\nabla \phi (x)\cdot y|^{p-2}y^{T}D^{2}\phi(x)y \,\frac{\dd y }{|y|^{d+sp}}+O(r^{p-1+(1-s)p}).$$
Using this estimate, \eqref{I1_est} and Lemma \ref{lem:niceidentity} we obtain \eqref{Id134}. 

Proceeding analogously we obtain \eqref{Id234}.
\end{proof}
 
 \subsection{Case $p\in (2,3)$.} As in the previous case, we can obtain a uniform error of order $p-2$ regardless of the value of the gradient. When the gradient is not zero, we consider two regions, as in the case $p\in(3,4)$, but this time the errors obtained are less than quadratic. This is due to two facts: $J_p'''$ is singular around zero; and the complementary region is not small enough.
 
  \begin{lemma}\label{Approx p greater 2}
Let $p\in (2,3)$. Consider a function $\phi\in C^4_b(B_R(x))$ for some $x\in \R^d$ and $0<R<1$. Then, for any $y\in B_{\frac{R}{2}}(0)\setminus\{0\}$ we have
\[
|D_y[\phi](x)- \frac{(p-1)}{|y|^p}|\nabla \phi(x)\cdot y|^{p-2}y^TD^2\phi(x) y|\leq C(p,d,\|\phi\|_{C^4_b(B_R(x))})|y|^{p-2}.
\]
 \end{lemma}

\begin{proof}
We perform a fourth-order Taylor expansion of $\phi$ and then a first-order expansion of $J_p$ at $x$ and $\nabla \phi(x)\cdot y$, respectively, to obtain 
\begin{equation}\label{Taylor1}
\begin{split}
|y|^{p}&D_y[\phi](x)=\,J_p(\phi(x+y)-\phi(x))+J_p(\phi(x-y)-\phi(x))\\ 
 = &\,J_p(\nabla \phi(x)\cdot y)+ J_p'(\Theta_1^+(y))( \phi_{23}^+(y)+O(|y|^{4} ))+J_p(-\nabla \phi(x)\cdot y)+ J_p'(\Theta_2^-(y))( \phi_{23}^-(y)+O(|y|^{4} )),
\end{split}
\end{equation}where $\Theta_j^{\pm}$ and $\phi_{23}^\pm$ were defined in \eqref{thetaj} and \eqref{phi23} respectively.. Since $J_p$ is odd, we cancel out the terms $J_p(\nabla \phi(x)\cdot y)$ and $J_p(-\nabla \phi(x)\cdot y)$. In \eqref{Taylor1} we add and substract the terms:
$$ J_p'(\pm\nabla \phi(x)\cdot y)(\phi_{23}^\pm(y)+O(|y|^{4} )),$$
to get
\begin{equation*}
\begin{split}
 |y|^{p}&D_y[\phi](x) =\, \dfrac{(p-1)}{|y|^{p}}|\nabla \phi(x)\cdot y|^{p-2}y^{T}D^{2}\phi(x)y + O(|y|^{p+2})  \\ 
& + \bigg[ J_p'(\Theta_1^+(y))- J_p'(\nabla \phi(x)\cdot y)\bigg]( \phi_{23}^+(y)+O(|y|^{4}))+ \bigg[ J_p'(\Theta_2^-(y))- J_p'(\nabla \phi(x)\cdot y)\bigg]( \phi_{23}^-(y)+O(|y|^{4})). 
\end{split}
\end{equation*}
From a slight modification of \cite[Lemma 2.4]{BaMe21}, we obtain for $p \in (2,3)$ that there is a constant $C_p >0$ such that the terms in brackets are of order $O(|y|^{2(p-1)})$, and the result follows.
\end{proof}

 \begin{lemma}\label{approx p 2 3}
Let $p\in (2,3)$ and $\beta\in(0,p-2)$. Consider a function $\phi\in C^4_b(B_R(x))$ for some $x\in \R^d$ and $0<R<1$, and let $K_0:=\|D^2\phi\|_{L^\infty(B_R(x))}$. Then, for any $y\in B_{\frac{R}{2}}(0)\setminus\{0\}$ in the region $|\nabla\phi(x)\cdot y|> K_0|y|^2$, we have
\[
|D_y[\phi](x)- \frac{(p-1)}{|y|^p}|\nabla \phi(x)\cdot y|^{p-2}y^TD^2\phi(x) y|\leq C(p,d,\|\phi\|_{C_b^4(B_R(x))})\Phi\left(\nabla \phi(x)\cdot \frac{y}{|y|}\right)|y|^{1+\beta},
\]
where $
\Phi(\eta):=1+|\eta|^{p-3}+|\eta|^{p-3-\beta}$.
 \end{lemma}

  \begin{proof}
 Given $f,g:\R^d\to\R$  smooth functions such that $f(y)\not=0$ for $y\not=0$ and $|g(y)/f(y)|<1$, we have
  \begin{equation*}
 \begin{split}
 J_p&(f(y)+g(y))=J_p(f(y))J_p(1+\frac{g(y)}{f(y)})\\
 &=J_p(f(y))\left(J_p(1) +  J_p'(1) \frac{g(y)}{f(y)}+ \frac{J_p''(1)}{2} \left(\frac{g(y)}{f(y)}\right)^2+ O\left([J_p'']_{C^\beta(B_1\setminus B_{1/2})}\left(\bigg|\frac{g(y)}{f(y)}\bigg|\right)^{2+\beta}\right) \right).
 \end{split}
\end{equation*}
Considering $f:=\phi_1$ and $g^\pm:=\phi_{23}^{\pm}$ as in \eqref{fgdef} the proof follows analogously to Lemma \ref{Approx p greater 3}. Noticing that, since $\beta<p-2$, we have that
  \[
  p-3>p-3-\beta>-1
  \]
we can use Lemma \ref{lem:intestim} to conclude that
\[
\begin{split}
|\phi_1(y)|^{p-3-\beta}O\left(|y|^{4+2\beta}\right)&=|\nabla \phi(x) \cdot \frac{y}{|y|}|^{p-3-\beta}|y|^{p-3-\beta} O(|y|^{4+2\beta})=|\nabla \phi(x) \cdot \frac{y}{|y|}|^{p-3-\beta} O(|y|^{p+1+\beta}). \qedhere
\end{split}
\]
  \end{proof}

 \begin{theorem}\label{teo_p23}
Let $p \in (2,3)$, $s\in(0,1)$ and $\beta\in(0,p-2)$. Consider a function $\phi\in C^4_b(B_R(x_0))$ for some $x_0\in \R^d$ and $0<R<1$, and such that $\nabla\phi\not=0$ in $\overline{ B_R(x_0)}$.  Let $R_0$ defined in \eqref{def:R0}. Then, for $r<\min\{R_0,R/2\}$ and $x\in B_{R/2}(x_0)$ we have that
\begin{equation}\label{Id123}
\,P.V.\int_{B_r}J_p(\phi(x+y)-\phi(x)) \frac{\dd y}{|y|^{d+sp}}= \frac{1}{2a_{s,p,d}}\Delta_p\phi(x) r^{p(1-s)} + O(r^{1+\beta+p(1-s)})
\end{equation}
with $a_{s,p,d}$ given in Lemma \ref{lem:niceidentity}, and
\begin{equation}\label{Id223}
\mathcal{M}_{r,i}^p[\phi](x)= \Delta_p\phi(x) + O(r^{1+\beta}),\quad i\in\{1,2\},
\end{equation}
where the error terms are uniform in $x$ and $\mathcal{M}_{r,i}^p$ is given by \eqref{Mlocal}.
\end{theorem}

\begin{proof}
We follow the same strategy as in the proof of Theorem \ref{theo_p34}. In fact, defining $I_1$ and $I_2$ as in \eqref{I1I2}, and using Lemma \ref{Approx p greater 2} and Lemma \ref{approx p 2 3} one can see that 
$$I_1=\frac{1}{2}(p-1)\int_{B_r \cap \Omega} 
|\nabla \phi (x)\cdot y|^{p-2}y^{T}D^{2}\phi(x)y \,\frac{\dd y }{|y|^{d+sp}}  +O(r^{1+\beta+(1-s)p}),$$
with $\beta\in (0,p-2)$ and
$$I_2=\frac{1}{2}(p-1)\int_{B_r \cap (\R^d\setminus\Omega)} 
|\nabla \phi (x)\cdot y|^{p-2}y^{T}D^{2}\phi(x)y \,\frac{\dd y }{|y|^{d+sp}}  +O(r^{p-1+(1-s)p}).$$
Identity \eqref{Id123} follows as a consequence of Lemma \ref{lem:niceidentity}. An analogous argument proves \eqref{Id223}.
\end{proof}

\subsection{Case $p\in(1,2)$.}
This case is more involved since we are not allowed to do a general Taylor expansion (we cannot differentiate $J_p$ even once) as in the previous cases. Nevertheless, we can give a precise estimate far from the singular region.
 \begin{lemma}\label{approx p 1 2}
Let $p\in (1,2)$, $\beta\in (0,p-1)$, and fix $\eta>0$. Consider a function $\phi\in C^4_b(B_R(x))$ for some $x\in \R^d$ and $0<R<1$. Then, for any $y\in B_{\frac{R}{2}}(0)\setminus\{0\}$ such that $|\nabla\phi(x)\cdot y|\geq \eta$, we have
\[
|D_y[\phi](x)- \frac{(p-1)}{|y|^p}|\nabla \phi(x)\cdot y|^{p-2}y^TD^2\phi(x) y|\leq C(p,d,\|\phi\|_{C^4_b(B_R(x))})\Phi\left(x,y,\phi\right)|y|^{\beta}
\]
where 
\begin{equation}\label{constant12}
\Phi(x,y,\phi)=|\nabla \phi(x) \cdot \frac{y}{|y|}+|y|\frac{y^T}{|y|}D^2\phi(x)  \frac{y}{|y|}|^{p-2-\beta}+|\nabla \phi(x) \cdot \frac{y}{|y|}|^{p-2-\beta} .
\end{equation} 
 \end{lemma}
For the proof we use some ideas from \cite{dTLi21}.
\begin{proof}
Expanding $\phi(x+y)$ it easily follows that
$$J_p(\phi(x+y)-\phi(x))=J_p(\nabla\phi(x)\cdot y+\frac{1}{2}y^TD^2\phi(x)y+O(|y|^3)),$$
and using \cite[Lemma 3.4]{KoKuLi19},
\begin{equation}\begin{split}\label{estOrder2}
|J_p(\nabla\phi(x)\cdot y+\frac{1}{2}y^TD^2\phi(x)y&+O(|y|^3))-J_p(\nabla\phi(x)\cdot y+\frac{1}{2}y^TD^2\phi(x)y)|\\
&\leq C\left(|\nabla\phi(x)\cdot y+y^TD^2\phi(x)y|+|O(|y|^3)|\right)^{p-2}O(|y|^{3})\\
&\leq C\left(|\nabla\phi(x)\cdot \frac{y}{|y|}+|y|\frac{y^T}{|y|}D^2\phi(x)\frac{y}{|y|}|\right)^{p-2}O(|y|^{p+1}).
\end{split}\end{equation}
Thus we will focus on estimating the term $J_p(\nabla\phi(x)\cdot y+\frac{1}{2}y^TD^2\phi(x)y)$. Since $|\nabla\phi(x)\cdot y|\geq \eta$,
we can do a Taylor expansion of $J_p$ to get, for every $\beta\in (0,p-1)$,
\begin{equation*}\begin{split}
|J_p(\nabla\phi(x)\cdot y+&\frac{1}{2}y^TD^2\phi(x)y)-J_p(\nabla\phi(x)\cdot y)-J'_p(\nabla\phi(x)\cdot y)y^TD^2\phi(x)y|\\
&\leq C\left(|\nabla\phi(x)\cdot y+y^TD^2\phi(x)y|^{p-2-\beta}+|\nabla\phi(x)\cdot y|^{p-2-\beta}\right)|y^TD^2\phi(x)y|^{1+\beta}\\
&\leq C(p,d,\|\phi\|_{C^4_b(B_R(x))})\Phi\left(x,y,\phi\right)|y|^{p+\beta},
\end{split}\end{equation*}
with $\Phi$ given in \eqref{constant12}. Proceeding analogously for the term arising from $J_p(\phi(x-y)-\phi(x))$ we can conclude that
\begin{equation*}\begin{split}
|y|^p D_y[\phi](x)&=J_p'(\nabla\phi(x)\cdot y)y^TD^2\phi(x)y+\Phi(x,y,\phi)O(|y|^{\beta})+O(|y|^{3-p}),
\end{split}\end{equation*}
and the result follows.
\end{proof}

 \begin{theorem}\label{teo_p12}
Let $p \in (1,2)$, $s\in(0,1)$ and $\beta\in(0,p-1)$. Consider a function $\phi\in C^4_b(B_R(x_0))$ for some $x_0\in \R^d$ and $0<R<1$, and such that $\nabla\phi\not=0$ in $\overline{ B_R(x_0)}$. Then, for $r<R/2$ and $x\in B_{R/2}(x_0)$ we have that
\begin{equation}\label{Id112}
\,P.V.\int_{B_r}J_p(\phi(x+y)-\phi(x)) \frac{\dd y}{|y|^{d+sp}}= \frac{1}{2a_{s,p,d}}\Delta_p\phi(x) r^{p(1-s)} + O(r^{\beta+p(1-s)}),
\end{equation}
with $a_{s,p,d}$ given in Lemma \ref{lem:niceidentity}, and
\begin{equation}\label{Id212}
\mathcal{M}_{r,i}^p[\phi](x)= \Delta_p\phi(x) + O(r^{\beta}),\quad i\in\{1,2\},
\end{equation}
where the error terms are uniform in $x$ and $\mathcal{M}_{r,i}^p$ is given by \eqref{Mlocal}.
\end{theorem}
 
 \begin{proof}
Fix $\eta>0$ and consider $\Omega:=\{y\in\R^d:\,|\nabla\phi(x)\cdot y|\leq \eta\}$. Define
\begin{equation}\begin{split}\label{I1I2p12}
I_1:=&\,\,P.V.\int_{B_r\cap \Omega}J_p(\phi(x+y)-\phi(x)) \frac{\dd y}{|y|^{d+sp}},\quad
I_2:=\,\,P.V.\int_{B_r\cap(\R^d\setminus\Omega)}J_p(\phi(x+y)-\phi(x)) \frac{\dd y}{|y|^{d+sp}}.
\end{split}\end{equation}
Using \eqref{estOrder2}, the symmetry of the domain and the oddness of $J_p$ we can write
\begin{equation*}\begin{split}
I_1=&\,\,\underbrace{P.V\int_{B_r\cap \Omega}(J_p(\nabla\phi(x)\cdot y+y^TD^2\phi(x)y)-J_p(\nabla\phi(x)\cdot y))\frac{\dd y}{|y|^{d+sp}}}_{I_{11}}\\
&\,+\underbrace{O\left(\int_{B_r\cap \Omega}\left(|\nabla\phi(x)\cdot \frac{y}{|y|}+|y|\frac{y^T}{|y|}D^2\phi(x)\frac{y}{|y|}|\right)^{p-2}\frac{|y|^{p+1}\dd y}{|y|^{d+sp}}\right) }_{I_{12}}.
\end{split}\end{equation*}
By a straightforward generalization of \cite[Lemma A3]{dTLi21} for the singular measure (in the spirit of Lemma \ref{lem:intestim}), and the Dominated Convergence Theorem, it can be proven that $I_{12}=o_\eta(1)$. Likewise, by \cite[Lemma 3.4]{KoKuLi19}, it follows that
$$|I_{11}|\leq C\int_{B_r\cap \Omega}\left(|\nabla\phi(x)\cdot \frac{y}{|y|}+|y|\frac{y^T}{|y|}D^2\phi(x)\frac{y}{|y|}|\right)^{p-2}\frac{|y|^{p}\dd y}{|y|^{d+sp}}=o_\eta(1).$$
By Lemma \ref{approx p 1 2},
\begin{equation*}\begin{split}
I_2=&\,\frac{1}{2}(p-1)\int_{B_r \cap (\R^d\setminus\Omega)} 
|\nabla \phi (x)\cdot y|^{p-2}y^{T}D^{2}\phi(x)y \,\frac{\dd y }{|y|^{d+sp}}  \\
&+O\left(\int_{B_r \cap (\R^d\setminus\Omega)} 
\Phi(x,y,\phi)\,\frac{|y|^{p+\beta}\dd y }{|y|^{d+sp}} \right),
\end{split}\end{equation*}
with $\Phi$ given in \eqref{constant12}, and $\beta\in (0,p-1)$. Notice that we can bound the last term by the integral in $B_r$ and, using once again \cite[Lemma A3]{dTLi21}  on the last integral, we conclude that
$$\bigg|I_2-\frac{1}{2}(p-1)\int_{B_r \cap (\R^d\setminus\Omega)} 
|\nabla \phi (x)\cdot y|^{p-2}y^{T}D^{2}\phi(x)y \,\frac{\dd y }{|y|^{d+sp}}\bigg|\leq C|y|^{\beta+(1-s)p},$$
where the last constant is independent of $\eta$. Therefore
\begin{equation*}\begin{split}
\,P.V\int_{B_r}J_p(\phi(x+y)-\phi(x)) \frac{\dd y}{|y|^{d+sp}}=&\,\frac{1}{2}(p-1)\int_{B_r \cap (\R^d\setminus\Omega)} 
|\nabla \phi (x)\cdot y|^{p-2}y^{T}D^{2}\phi(x)y \,\frac{\dd y }{|y|^{d+sp}}\\
&\,+O(|y|^{\beta+p})+o_\eta(1).
\end{split}\end{equation*}
Letting $\eta\to 0$ and applying Lemma \ref{lem:niceidentity} we obtain \eqref{Id112}. Identity \eqref{Id212} can be analogously obtained.
 \end{proof}
  
 \subsection{Results in dimension $d=1$.}
 We first note that, in this case, 
 \[
 \Delta_p\phi(x)=(p-1) |\phi_x(x)|^{p-2} \phi_{xx}(x)
 \]
 and for $r>0$, $D_r[\phi]$ defined by \eqref{defDy} is precisely the asymptotic expansion for the $p$-Laplacian given by \eqref{eq:asexp1D}. We have the following result.
  \begin{lemma}\label{1d}
Let $p> 1$. Consider a function $\phi\in C^4_b(B_R(x))$ for some $x\in \R^d$ and $0<R<1$ and such that $\phi_x\not=0$ in $\overline{B_R(x)}$ and let
\[
r<R_0:=\frac{\inf_{B_R(x)}{|\phi_x|}}{\|\phi_{xx}\|_{L^\infty(B_R(x))}}.
\]
Then, for any $r<\min\{R_0,R/2\}$ we have
\[
|D_r[\phi](x)- \Delta_p \phi(x)|\leq C(p,d,\|\phi\|_{C^4_b(B_R(x))}) (\inf_{B_{R}(x)}| \phi_x|^{p-4}+1)r^2.
\]
 \end{lemma}
 \begin{proof}
 The result for $p\geq 4$ is precisely given by Lemma \ref{Approx p greater 4}. For $p\in(1,4)$ we follow the proof of Lemma \ref{Approx p greater 3}. We note that, in this case, the fact that the gradient does not vanish, directly implies that $\nabla \phi(x)\cdot y= \phi_x(x)r\not=0$ since there are no ortogonal directions. 
 We then take
 \[
 f(r)= \phi_x(x) r
 \] 
 and
 \[
 g^\pm(r)= \phi_{xx}(x) \frac{r^2}{2} \pm \phi_{xxx}(x)  \frac{r^3}{3!} + O(r^4) = \phi_{xx}(\xi_{\pm}) \frac{r^2}{2}
 \]
 for some $\xi_{\pm}\in B_r(x)$. With this in mind, we note that
 \[
\left| \frac{g^\pm(r)}{f(r)}\right| \leq \frac{\|\phi_{xx}\|_{L^\infty(B_R(x))} \frac{r^2}{2}}{\inf_{B_r(x) } \{|\phi_x|\} r} <\frac{1}{2} 
 \]
 as long as $r<R<\inf_{B_r(x) }\{|\phi_x|\}/\|\phi_{xx}\|_{L^\infty(B_R(x))}$. The rest of the proof follows line by line the proof of Lemma \ref{Approx p greater 3}. 
 \end{proof}

\subsection{Proofs of the asymptotic expansions in Theorem \ref{asympExpUnif}, Theorem \ref{asympExp} and Theorem \ref{thm:1Dresult} }
We start with the identities for the $p$-Laplacian. Observe that \eqref{LocalGrad} is already proven for every range of $p$ in Theorems \ref{teo4}, \ref{theo_p34}, \ref{teo_p23}, and \ref{teo_p12}. In the case $p\geq4$, \eqref{LocalUnif} is contained in Theorem \ref{teo4}. The case $p\in(2,4)$ can be analogously obtained  applying Lemmas \ref{p greater 3} and \ref{Approx p greater 2}.

In the fractional $p$-Laplacian case we start by splitting the operator in the singular and nonsingular part as follows 
$$-(-\Delta)^s_p\phi (x)= P.V.\int_{B_r}J_p(\phi(x+y)-\phi(x))\frac{\dd y}{|y|^{d+sp}}+\int_{\R^d\setminus B_r}J_p(\phi(x+y)-\phi(x))\frac{\dd y}{|y|^{d+sp}}.$$
Then \eqref{NonLocalGrad} is contained in Theorems \ref{teo4}, \ref{theo_p34}, \ref{teo_p23} and \ref{teo_p12}. In the case $p\geq 4$, \eqref{NonLocalUnif} has been proven in Theorem \ref{teo4}, and the rest of the cases can be deduced in the same way from Lemmas \ref{p greater 3} and \ref{Approx p greater 2}.

Finally, Theorem \ref{thm:1Dresult} for $\mathcal{M}_{r,1}^p$ is precisely Lemma \ref{1d} and the rest follows by integration as before.
  
 \section{Discretizations of the fractional $p$-Laplacian}\label{sec:disc}
 In this section we will prove Theorem \ref{DiscretOrder}. We use the notation of subsection \ref{subsec:disc}. Assume throughout the whole section that $\phi\in C^4_b(B_R(x))\cap L^\infty(\R^d)$, $x\in \R^d$, for some $R>0$. We will analyze separately the error of the discretizations/quadratures near and far from the origin of the integration domain. Let $r>0$ and denote
$$A_r(x):=\frac{1}{2}\int_{B_r}(J_p(\phi(x+y)-\phi(x))+J_p(\phi(x-y)-\phi(x)))\dd y,$$
and
$$\tilde{A}_r(x):=\frac{1}{2}\int_{\tilde{B}_r}(J_p(\phi(x+y)-\phi(x))+J_p(\phi(x-y)-\phi(x)))\dd y,$$
where 
$$\tilde{B}_r:=\bigcup_{y_\alpha\in B_r} Q_\alpha \quad \textup{with} \quad Q_{\alpha}:=y_\alpha + \frac{h}{2}[-1,1)^d.$$ 
Define also 
$$\tilde{A}_r^h(x):=h^d\sum_{y_\alpha\in B_r}J_p(\phi(x+y_\alpha)-\phi(x)),$$
which is just part of the first term of \eqref{Discret}.
\begin{lemma}\label{lem:ballorig}
Assume $h\leq r/4$. Then 
\[
|A_r(x)-\tilde{A}_r(x)|  = O(r^{p+d-1}h),
\]
and
$$|\tilde{A}_r(x)-\tilde{A}_r^h(x)|=\begin{cases}
O(r^{p-3}h^2r^d)\;\mbox{ if } p>3,\\
O(r^{p-2}h^2r^d+h^{p-1}r^d)\;\mbox{ if }p\in (2,3],\\
O(h^{p-1}r^d)\;\mbox{ if }p\in (1,2).
\end{cases}$$
\end{lemma}
\begin{proof}
We follow the same strategy of \cite[Proof of Theorem 1.1]{dTLi22}, but computing the precise asymptotic orders. Actually, it can be seen there that 
\begin{equation*}
\begin{split}|A_r(x)-\tilde{A}_r(x)| & = O(r^{p+d-1}h).
\end{split}
\end{equation*} 
Let us estimate now $|\tilde{A}_r(x)-\tilde{A}_r^h(x)|$. Since $h^d=|Q_\alpha|$ we can write
\begin{equation}\begin{split}\label{difTilde}
|\tilde{A}_r&(x)-\tilde{A}_r^h(x)|\\
&=\frac{1}{2}\bigg|\sum_{y_\alpha\in B_r}\int_{Q_0}(J_p(\phi(x+y_\alpha+y)-\phi(x))+J_p(\phi(x+y_\alpha-y)-\phi(x))-2J_p(\phi(x+y_\alpha)-\phi(x)))\dd y\bigg|,
\end{split}\end{equation}
and it can be seen that
\begin{equation*}\begin{split}
|J_p(\phi(x+y_\alpha+y)&-\phi(x))+J_p(\phi(x+y_\alpha-y)-\phi(x))-2J_p(\phi(x+y_\alpha)-\phi(x))|\\
&=\begin{cases}
O(r^{p-3}h^2)\;\mbox{ if } p>3,\\
O(r^{p-2}h^2+h^{p-1})\;\mbox{ if }p\in (2,3],\\
O(h^{p-1})\;\mbox{ if }p\in (1,2).
\end{cases}
\end{split}\end{equation*}
The result follows by replacing this information in \eqref{difTilde} and noticing that, since $h=O(r)$, we have $\sum_{y_\alpha\in B_r} h^d=|\tilde{B}_r|\leq |B_{r+\sqrt{d}h}|=O(r^d)$.
\end{proof}
Let us estimate now the error far from the origin. Define
$$T_r^s(x):=\int_{\R^d\setminus B_r}J_p(\phi(x+y)-\phi(x)) \frac{\dd y}{|y|^{d+sp}},$$
\[
\tilde{T}_r^s(x):=\frac{1}{2}\int_{\R^d \setminus\tilde{B}_r}(J_p(\phi(x+y)-\phi(x))+J_p(\phi(x-y)-\phi(x))) \frac{\dd y}{|y|^{d+sp}},
\]
and
$$\tilde{T}_{r,j}^{s,h}(x):=\sum_{y_\alpha\in \R^d\setminus B_r}J_p(\phi(x+y_\alpha)-\phi(x))W_{\alpha,i},\quad i=1,2,$$
with $W_{\alpha,i}$ defined in \eqref{def:Wj}.
\begin{lemma}\label{estT1}
Assume $h\leq r/4$. Then
\[
|T_r^s(x)-\tilde{T}_r^s(x)| =O(h r^{p-1-sp})
\]
and 
$$|\tilde{T}_r^s(x)-\tilde{T}_{r,1}^{s,h}(x)|=\begin{cases}
O(h) \;\mbox{ if }p\geq 2 \mbox{ and } p\geq \frac{2}{1-s},\\
O(hr^{p-2-sp})\;\mbox{ if }p\geq 2 \mbox{ and } p< \frac{2}{1-s},\\
O(h^{p-1}r^{-sp})\;\mbox{ if }1<p<2.
\end{cases}$$
\end{lemma}
\begin{proof}
Similarly to the proof of Lemma \ref{lem:ballorig}, we get
\[
|T_r^s(x)-\tilde{T}_r^s(x)| =O(h r^{p-1-sp}).
\]
On the other hand, we can write
$$|\tilde{T}_r^s(x)-\tilde{T}_{r,1}^{s,h}(x)|=\bigg|\sum_{y_\alpha\in \R^d\setminus B_r} \int_{Q_\alpha}(J_p(\phi(x+y)-\phi(x))-J_p(\phi(x+y_\alpha)-\phi(x))) \frac{\dd y}{|y|^{d+sp}}\bigg|.$$
If $p\geq 2$, using the regularity and the boundedness of $\phi$ we have, for every $y_\alpha$ and every $y\in Q_\alpha$,
$$|J_p(\phi(x+y)-\phi(x))-J_p(\phi(x+y_\alpha)-\phi(x))|\leq C|\phi(x+y)-\phi(x+y_\alpha)|\leq C|y-y_\alpha|\leq Ch.$$
Thus, we can deduce that
\begin{equation}\begin{split}\label{T1}
\bigg|\sum_{y_\alpha\in \R^d\setminus B_1} \int_{ Q_\alpha}&(J_p(\phi(x+y)-\phi(x))-J_p(\phi(x+y_\alpha)-\phi(x))) \frac{\dd y}{|y|^{d+sp}}\bigg|\\
&\leq Ch\sum_{y_\alpha\in \R^d\setminus B_1} \int_{Q_\alpha}\frac{\dd y}{|y|^{d+sp}}\leq Ch\int_{\R^d\setminus B_{1/2}}\frac{\dd y}{|y|^{d+sp}}\leq Ch.
\end{split}\end{equation}
To estimate the case $y_\alpha\in B_1\setminus B_r$, observe that if $h\leq r$ then, for every $y\in (B_1\setminus B_r)\cap Q_\alpha$, there exists a constant $c_0$, depending only on the dimension, such that $|y|\geq c_0h$, and hence
$$|\phi(x+y_\alpha)-\phi(x)|\leq C|y_\alpha|\leq C(|y|+h)\leq C|y|.$$
Therefore, we can use  \cite[Lemma A.1]{dTLi22} to get
\begin{equation}\begin{split}\label{T2}
\bigg|\sum_{y_\alpha\in B_1\setminus B_r} &\int_{ Q_\alpha}(J_p(\phi(x+y)-\phi(x))-J_p(\phi(x+y_\alpha)-\phi(x))) \frac{\dd y}{|y|^{d+sp}}\bigg|\\
&\leq Ch\sum_{y_\alpha\in B_1\setminus B_r} \int_{ Q_\alpha}(|\phi(x+y)-\phi(x)|^{p-2}+|\phi(x+y_\alpha)-\phi(x)|^{p-2}) \frac{\dd y}{|y|^{d+sp}}\\
&\leq Ch\int_{ B_{3/2}\setminus B_{r/2}}\ |y|^{p-2}\frac{\dd y}{|y|^{d+sp}}\leq Ch(1+r^{p-2-sp}).
\end{split}\end{equation}
Putting together \eqref{T1} and \eqref{T2} the result for $p\geq 2$ follows.

If $p\in(1,2)$ we can use the H\"older continuity of $J_p$ to deduce, for every $y_\alpha$ and every $y\in Q_\alpha$,
$$|J_p(\phi(x+y)-\phi(x))-J_p(\phi(x+y_\alpha)-\phi(x))|\leq C|\phi(x+y)-\phi(x+y_\alpha)|^{p-1}\leq Ch^{p-1}$$
and hence 
\begin{equation*}\begin{split}
\bigg|\sum_{y_\alpha\in \R^d\setminus B_r} &\int_{ Q_\alpha}(J_p(\phi(x+y)-\phi(x))-J_p(\phi(x+y_\alpha)-\phi(x))) \frac{\dd y}{|y|^{d+sp}}\bigg|\\
&\leq Ch^{p-1}\sum_{y_\alpha\in \R^d\setminus B_r} \int_{ Q_\alpha}\frac{\dd y}{|y|^{d+sp}}\leq Ch^{p-1}\int_{\R^d\setminus B_{r/2}}\frac{\dd y}{|y|^{d+sp}}\leq Ch^{p-1}(1+r^{-sp}).
\end{split}\end{equation*}
\end{proof}
Let us analyze now the second discretization.
\begin{lemma}\label{estT2}
Assume $h\leq \frac{r}{4\sqrt{d}}$. Then 
$$|\tilde{T}_r^s(x)-\tilde{T}^{s,h}_{r,2}(x)|=\begin{cases}
O(h) \;\mbox{ if }p\geq 2 \mbox{ and }p\geq \frac{2}{1-s},\\
O(hr^{p-2-sp})\;\mbox{ if }p\geq 2\mbox{ and }p< \frac{2}{1-s},\\
O(h^{p-1}r^{-sp}+hr^{p-2-sp})\;\mbox{ if }1<p<2.
\end{cases}$$
\end{lemma}
\begin{proof}
Since $h^d=|Q_\alpha|$ for every $\alpha$, we can write
\begin{equation*}\begin{split}
|\tilde{T}_r^s(x)-\tilde{T}^{s,h}_{r,2}(x)|=&\,\bigg|\sum_{y_\alpha\in \R^d\setminus B_r}\int_{Q_\alpha}\left(\frac{J_p(\phi(x+y)-\phi(x))}{|y|^{d+sp}}-\frac{J_p(\phi(x+y_\alpha)-\phi(x))}{|y_\alpha|^{d+sp}}\right)\dd y\\
=&\,\bigg|\underbrace{\sum_{y_\alpha\in \R^d\setminus B_r}\int_{Q_\alpha}\left(\frac{J_p(\phi(x+y)-\phi(x))-J_p(\phi(x+y_\alpha)-\phi(x))}{|y|^{d+sp}}\right)\dd y}_{I_1}\\
&+\underbrace{\sum_{y_\alpha\in \R^d\setminus B_r}\int_{Q_\alpha}J_p(\phi(x+y_\alpha)-\phi(x))\left(\frac{1}{|y|^{d+sp}}-\frac{1}{|y_\alpha|^{d+sp}}\right)\dd y}_{I_2}\bigg|.
\end{split}\end{equation*}
Notice that $I_1$ is $\tilde{T}_r^s(x)-\tilde{T}_{r,1}^{s,h}(x)$ and Lemma \ref{estT1} provides the bound. We split $I_2$ into two regions,
$$I_2=\underbrace{\sum_{y_\alpha\in \R^d\setminus B_1}\int_{ Q_\alpha}\ldots}_{I_{21}}\;+\underbrace{\sum_{y_\alpha\in B_1\setminus B_r}\int_{Q_\alpha}}_{I_{22}}\ldots$$
Notice that if $y\in Q_\alpha$ it can be written as $y=y_\alpha +\frac{h}{2}v$, with $|v|<\sqrt{d}$, and hence
\begin{equation}\label{yyalpha}
\left|\frac{1}{|y|^{d+sp}}-\frac{1}{|y_\alpha|^{d+sp}}\right| \leq C\frac{ 1}{\min_{z\in Q_\alpha}|z|^{d+sp+1}}|y-y_\alpha| \leq C \frac{h}{|y|^{d+sp+1}}
\end{equation}
Let us consider first the case $y_\alpha\in \R^d\setminus B_1$. Using the boundedness of $\phi$ and \eqref{yyalpha} we get
\begin{equation*}\begin{split}
|I_{21}|&\leq C\sum_{y_\alpha\in \R^d\setminus B_1}\int_{ Q_\alpha}\frac{h}{|y|^{d+sp+1}}\dd y\leq Ch\int_{\R^d\setminus B_{1/2}}\frac{\dd y}{|y|^{d+sp+1}}\leq Ch.
\end{split}\end{equation*}
Likewise, using that
$|J_p(\phi(x+y_\alpha)-\phi(x))|\leq C(|y|+h)^{p-1}\leq C|y|^{p-1},$
and \eqref{yyalpha} we see that
$$|I_{22}|\leq Ch\int_{B_{3/2}\setminus B_{r/2}}\frac{|y|^{p-1}}{|y|^{d+1+sp}}\dd y\leq Ch(1+r^{p-2-sp}).$$
Therefore
$$|I_2|=\begin{cases}
O(h) \;\mbox{ if }p\geq \frac{2}{1-s},\\
O(hr^{p-2-sp})\;\mbox{ if }p< \frac{2}{1-s},
\end{cases}$$
and the result follows.
\end{proof}

\begin{proof}[Proof of Theorem \ref{DiscretOrder}]
By either Theorem \ref{asympExpUnif}, Theorem \ref{asympExp} or Theorem \ref{thm:1Dresult} we have
\begin{equation*}\begin{split}
|(-\Delta)^{s,h}_{p,i}\phi(x)-(-\Delta)^s_p\phi(x)|&\leq |(-\Delta)^{s,h}_{p,i}\phi(x)-\mathcal{M}_r^{s,p}[\phi](x)|+|\mathcal{M}_r^{s,p}[\phi](x)-(-\Delta)^s_p\phi(x)|,\\
&=|(-\Delta)^{s,h}_{p,i}\phi(x)-\mathcal{M}_r^{s,p}[\phi](x)|+O(r^{\gamma+p(1-s)}),\quad i=1,2,
\end{split}\end{equation*}
as $r\to 0^+$, with $\gamma$ given in either \eqref{def:gamma} or \eqref{def:nu}. Notice that 
\begin{equation*}\begin{split}
|(-\Delta)^{s,h}_{p,i}&\phi(x)-\mathcal{M}_r^{s,p}[\phi](x)|\leq \frac{p+d}{p(1-s)r^{d+sp}}|A_r(x)-\tilde{A}_r^h(x)|+|T_r^s(x)-\tilde{T}_{r,j}^{s,h}(x)|\\
&\leq C\left( \frac{1}{r^{d+sp}}|A_r(x)-\tilde{A}_r(x)|+ \frac{1}{r^{d+sp}}|\tilde{A}_r(x)-\tilde{A}_r^h(x)|+|T_r^s(x)-\tilde{T}_r^s(x)|+|\tilde{T}_r^s(x)-\tilde{T}_{r,j}^{s,h}(x)|\right),
\end{split}\end{equation*}
where $C$ is a positive constant depending on $d$, $p$ and $s$. Applying Lemma \ref{estT1} and Lemma  \ref{estT2} we conclude the proof with a suitable choice of $r\asymp h^\mu$ with $\mu\in(0,1]$.
\end{proof}

\subsection{Precise orders in terms of $h$ and $r$}\label{subsec:orders}
We estimate precisely 
\[
E:=|(-\Delta)^{s,h}_{p,i}\phi(x)-(-\Delta)_p^s\phi(x)|
\]
according to the errors of the asymptotic expansion proved in Theorem \ref{asympExpUnif}, Theorem \ref{asympExp}. Similar results can be obtained in dimension $d=1$ using Theorem \ref{thm:1Dresult}.

Let $\gamma$ be given by either \eqref{def:gamma}, \eqref{def:nu}, and assume $r\asymp h^\mu$ for a suitable choice of $\mu\in(0,1]$. We distinguish three ranges:
\begin{itemize}
\smallskip

\item[(i)]\textbf{Case $p>3$.} We have
\[
E=O(r^{p(1-s)-3}h^2+ r^{\min\{0,p(1-s)-2\}}h + r^{\gamma+p(1-s)})
\] 
We observe that:

If $p(1-s)\geq2$, we have convergence for every $ \mu  \leq1$. In particular, taking $\mu=1$ we get $E=O(h)$, which is computationally efficient (first order discretization in $h$).

If $p(1-s)<2$, convergence is ensured for $\mu<1/(2-p(1-s))$. In particular, the optimal choice $\mu=\frac{1}{\gamma +2}$ gives $E=O(h^{\frac{\gamma+p(1-s)}{\gamma+2}})$, which is always sublinear but at least $E=O(h^{\frac{1}{3}-\delta})$ for all $\delta>0$. When the gradient does not vanish, we get $E=O(h^{\frac{1}{2}-\delta})$.

\smallskip

\item[(ii)]\textbf{Case $p\in(2,3]$.} Here

\[
E=O(h^{p-1}r^{-sp}+ r^{\min\{0,p(1-s)-2\}}h + r^{\gamma+p(1-s)})
\] 
Notice that:

If $p(1-s)\geq2$, we have convergence for every $\mu<(p-1)/(sp)$. In particular, the  choice $\mu=\frac{1}{\gamma+p(1-s)}$ leads to $E=O(h)$.

If $p(1-s)<2$, convergence is ensured for $\mu<1/(2-p(1-s))$. In particular, the optimal choice $\mu=\frac{1}{\gamma +2}$ gives $E=O(h^{\frac{\gamma+p(1-s)}{\gamma+2}})$. Away from the zero gradient points this is at least $E=O(h^{\frac{1}{3}-\delta})$ for all $\delta >0$.
\smallskip

\item[(iii)]\textbf{Case $p\in(1,2)$.} We obtain
\[
E=O(h^{p-1}r^{-sp}+ r^{\gamma+p(1-s)})
\]
Convergence is ensured for $\mu<(p-1)/(sp)$. The choice $\mu=(p-1)/(\gamma+p)$ leads to $E=O(h^{(p-1)\left(1-\frac{sp}{\gamma+p}\right)})$.
\end{itemize}
\smallskip

 As an illustrative example, we consider the function $\phi:\R\to \R$ given by $\phi(x)=\max\{0,x\}^s$ and run numerical simulations using the above discretization in certain range of the parameters $p$ and $s$. It is known that for all $s\in(0,1)$ and all $p> 2$ we have that $(-\Delta)^s_p\phi(x)=0$ for all $x>0$ (see \cite[Lemma 3.1]{IaMoSq16}). To run the simulation, we truncate the function $\phi$ for $|x|$ large enough so that the truncation error does not interfere with the numerical error. We present the results in Figure \ref{fig:p(1-s)mayor2pmayor3}.

\begin{figure}[h!]
         \centering
         \includegraphics[width=\textwidth]{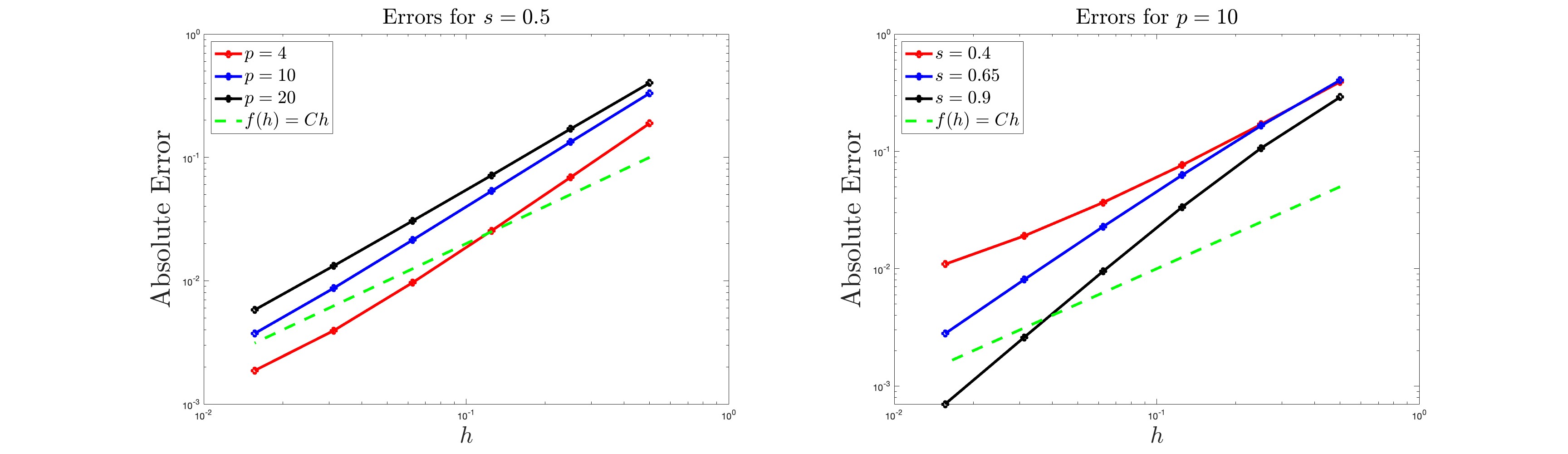}
         \caption{$p(1-s)\geq2$ and $p\geq3$}
         \label{fig:p(1-s)mayor2pmayor3}
  \end{figure}

 \section{Application to a parabolic problem}\label{sec:app}

 In this section we give an explicit finite difference numerical scheme to solve the parabolic problem \eqref{parabolic problem} for H\"{o}lder continuous data. We will apply the stability conditions \eqref{as:cfl} together with the consistency result Theorem \ref{DiscretOrder} to get convergence of the schemes to continuous viscosity solutions of \eqref{parabolic problem}. 
 Throughout the section, we will follow the notation of subsections \ref{subsec:disc} and \ref{subs: aplication}.

 We start by proving a technical result about the weights $\omega_\alpha$. For such purposes, given $\nu>0$, let us define
 \[
 S_\nu(r):=\left\{\begin{split}
 r^{-sp+\nu} \quad &\textup{if} \quad \nu\in (0, sp) \\
 |\log(r)| \quad &\textup{if} \quad \nu= sp \\
 1 \quad &\textup{if} \quad \nu> sp. 
 \end{split}\right.
 \]
 \begin{lemma}\label{lem:propwe}
 Let $p>2$, $s\in(0,1)$ and $h,r\in(0,1)$ such that $h\leq\frac{\sqrt{d}}{2}r$. Assume \eqref{as:weights} and let $\LL_h$ be given by \eqref{eq:Lh}. Then
 \begin{enumerate}[\rm (a)]
 \item \label{lem:propwe-item1}  The weights are symmetric and positive, i.e., $\omega_\alpha=\omega_{-\alpha}\geq0$ for all $\alpha\in \Z\setminus\{0\}$.
\item \label{lem:propwe-item2} The following summability properties hold: there exists a constant $C_{s,p,d}\geq 1  $ such that
 \[
 \sum_{\alpha\not=0} \omega_\alpha\leq \frac{C_{s,p,d}}{r^{sp}}, \quad   \sum_{|y_\alpha|\geq1} \omega_{\alpha} \leq C_{s,p,d} \quad \text{and}  \quad \sum_{0<|y_\alpha|<1}|y_\alpha|^{\nu} \omega_\alpha\leq C_{s,p,d}S_\nu(r).
 \]
 \end{enumerate}
 \end{lemma}
 \begin{proof}
 Part \eqref{lem:propwe-item1} follows by construction. We prove part \eqref{lem:propwe-item2} for the weights corresponding to $W_{\alpha,1}$, i.e.,
\begin{equation*}
\omega_{\alpha}=  \left\{ \begin{split}
\frac{(p+d)}{p(1-s)} \frac{h^d}{ r^{d+sp}} \quad &\textup{if} \quad |y_\alpha|<|r|\\
\int_{ Q_\alpha}\frac{\dd y}{|y|^{d+sp}} \quad  &\textup{if} \quad |y_\alpha|\geq|r|.
\end{split}\right.
\end{equation*}
The case $W_{\alpha,2}$ follows similarly. Recall the estimate
\[
\sum_{y_\alpha \in B_r}h^{d} \leq |B_{r+\sqrt{d}h}|\leq C r^{d}.
\]
Then, 
\[
\sum_{\alpha\not=0} \omega_\alpha \leq \frac{C}{r^{d+sp}} \sum_{0<|y_\alpha|<r} h^d + \sum_{|y_\alpha|\geq r} \int_{ Q_\alpha}\frac{\dd y}{|y|^{d+sp}} \leq \frac{C}{r^{d+sp}}  |B_{r+\sqrt{d}h}| + \int_{\R^d\setminus B_{r/2}}\frac{\dd y}{|y|^{d+sp}} \leq C r^{-sp}.
\]
Moreover,
\[
\sum_{|y_\alpha|\geq 1} \omega_\alpha = \sum_{|y_\alpha|\geq 1} \int_{ Q_\alpha}\frac{\dd y}{|y|^{d+sp}} \leq \int_{\R^d\setminus B_{1/2}}\frac{\dd y}{|y|^{d+sp}}\leq C.
\]
To prove the last part, we first note that
\[
\sum_{0<|y_\alpha|<r}|y_\alpha|^{\nu} \omega_\alpha= \frac{C}{r^{d+sp}} \sum_{0<|y_\alpha|<r}|y_\alpha|^{\nu} h^d \leq  \frac{C}{r^{d+sp-\nu}}  |B_{r+\sqrt{d}h}| \leq C r^{-sp+\nu}.
\]
On the other hand, we can use the fact that there exists a constant $C$ such that $|y_\alpha|\leq C|y|$ for all $y\in Q_\alpha$ to get
\[
\sum_{r\leq|y_\alpha|<1}|y_\alpha|^{\nu} \omega_\alpha\leq C\sum_{r\leq|y_\alpha|<1}\int_{Q_\alpha}\frac{\dd y}{|y|^{d+sp-\nu}} \leq C\int_{B_1\setminus B_{r/2}} \frac{\dd y}{|y|^{d+sp-\nu}}.
\]
The conclusion follows.
 \end{proof}

The next result accounts for the existence and uniqueness of solutions to \eqref{scheme}.
\begin{lemma}Let $h,\tau,r\in(0,1)$ and $p>2$. Assume \eqref{as:weights} and \eqref{as:data}. Then there exists a unique solution $U \in l^{\infty}(\mathcal{G}_h \times \mathcal{T}_\tau)$ of  the numerical scheme \eqref{scheme}.
\end{lemma}
\begin{proof}
Using assumption \eqref{as:weights}  and by Lemma \ref{lem:propwe}, we have  the next bound for any $\psi \in l^{\infty}(\mathcal{G}_h)$:
\begin{equation}\label{shceme1}
\begin{split}
|\LL_h\psi(x_\alpha)| \leq \|\psi\|_{l^{\infty}(\Grid)}^{p-1} \sum_{\alpha\not=0} \omega_\alpha \leq  \frac{C}{r^{sp}} \|\psi\|_{l^{\infty}(\Grid)}^{p-1}.\end{split}
\end{equation}
Since $u_0$ is continuous and bounded, then $U^{0} \in l^{\infty}(\mathcal{G}_h)$.
 Applying \eqref{shceme1} we can solve problem \eqref{scheme} to find $U^1\in l^{\infty}(\mathcal{G}_h)$. Iterating, we can solve the problem and get $U^j\in l^{\infty}(\mathcal{G}_h)$ for any $j\in \N$. Uniqueness follows by construction.
 \end{proof}

 For further reference, we point out the following relation concerning assumption \eqref{as:cfl}: 
  \begin{equation}\label{relation cfl tau}
{\tau  S_{a(p-2)}(r)\leq K_{s,p,d}}.
 \end{equation}
 where $K_{s,p,d}$ is given by \eqref{eq:constantCFL}. 
In the next lemma, we state the uniform boundedness and uniform continuity of $U^{j}$ in space. 

\begin{lemma}\label{lem:equi}Let $h,\tau,r\in(0,1)$ and $p>2$. Assume \eqref{as:weights}, \eqref{as:data} and \eqref{as:cfl}. Let $U$ be the solution of \eqref{scheme}. Then, for every $j=0,..., N$, we have
\begin{equation}\label{boundedness space}
\sup_{\alpha \in \Z^d}|U_\alpha^{j}| \leq \|u_0\|_{L^\infty(\R^d)} + t_j\|f\|_{L^\infty(\R^d)},
\end{equation}
and
\begin{equation}\label{scheme2}
|U_\alpha^{j}-U_\gamma^{j}| \leq \Lambda_{u_0}(|x_\alpha-x_\gamma|)+ t_j \Lambda_f(|x_\alpha-x_\gamma|)\quad \textup{for all} \quad \alpha,\gamma \in \Z^d.
\end{equation}

\end{lemma}

\begin{proof}
The boundedness and the H\"{o}lder continuity of $u_0$ imply  \eqref{boundedness space} and \eqref{scheme2} for $j=0$. Suppose that \eqref{boundedness space}  and \eqref{scheme2} hold for some $j$. Then,

\begin{equation}\label{scheme7}
\begin{split}
U_\alpha^{j+1} &= U_\alpha^{j}+\tau \sum_{\beta \not=0}J_p(U_{\alpha+\beta}^{j}-U_{\alpha}^{j})\omega_\beta+ \tau f_\alpha \\ &  = U_\alpha^{j}\left( 1-\tau  \sum_{\beta \not=0} |U_{\alpha+\beta}^{j}-U_{\alpha}^{j}|^{p-2}\omega_\beta \right) + \tau  \sum_{\beta \not=0} |U_{\alpha+\beta}^{j}-U_{\alpha}^{j}|^{p-2}U_{\alpha+\beta}^j\omega_\beta+ \tau f_\alpha. 
\end{split}
\end{equation}
Note that, by the induction hypothesis and  Lemma \ref{lem:propwe}, we have
\begin{equation}\label{bound_yalpha}
\begin{split}
\sum_{\beta \not=0} |U_{\alpha+\beta}^{j}-U_{\alpha}^{j}|^{p-2}\omega_\beta &=\sum_{0<|y_\alpha|<1} |U_{\alpha+\beta}^{j}-U_{\alpha}^{j}|^{p-2}\omega_\beta + \sum_{|y_\alpha|\geq 1} |U_{\alpha+\beta}^{j}-U_{\alpha}^{j}|^{p-2}\omega_\beta\\
& \leq 2^{p-2} (L_{u_0}+ t_j L_f)^{p-2} \left(\sum_{0<|y_\alpha|<1}|y_\beta|^{a(p-2)}\omega_\beta+ \sum_{|y_\alpha|\geq 1} \omega_\beta\right) \\
&\leq 2^{p-1}C_{s, p,d}(L_{u_0}+ T L_f)^{p-2}S_{a(p-2)}(r),
\end{split}
\end{equation}
where $C_{s, p, d}$ comes from Lemma \ref{lem:propwe}. Then, by \eqref{as:cfl} and \eqref{relation cfl tau}, 
\[
 1-\tau  \sum_{\beta \not=0} |U_{\alpha+\beta}^{j}-U_{\alpha}^{j}|^{p-2}\omega_\beta\geq0.
\]
In this way, from \eqref{scheme7}, we have
\[
\begin{split}
\sup_{\alpha\in \Z^d}|U_\alpha^{j+1}| &\leq \sup_{\alpha\in \Z^d}|U_\alpha^{j}|\left( 1-\tau  \sum_{\alpha\not=0} |U_{\alpha+\beta}^{j}-U_{\alpha}^{j}|^{p-2}\omega_\beta \right) + \tau  \sup_{\alpha\in \Z^d}|U_\alpha^{j}|  \sum_{\alpha\not=0} |U_{\alpha+\beta}^{j}-U_{\alpha}^{j}|^{p-2}\omega_\beta+ \tau \sup_{\alpha\in \Z^d}|f_\alpha |\\
&=\sup_{\alpha\in \Z^d}|U_\alpha^{j}| + \tau \sup_{\alpha\in \Z^d}|f_\alpha |\\
&\leq \|u_0\|_{L^\infty(\R^d)} + t_{j+1}\sup_{\alpha\in \Z^d}|f_\alpha |,
\end{split}
\]
which shows  \eqref{boundedness space} for $j+1$.

Let us prove \eqref{scheme2} for $j+1$.
Now, 
\begin{equation*}
\begin{split}
U_\alpha^{j+1}-U_\gamma^{j+1}= U_\alpha^{j}-U_\gamma^{j} + \tau \left[\LL_hU_\alpha^{j}-\LL_hU_\gamma^{j} \right]+\tau (f_\alpha-f_\gamma).
\end{split}
\end{equation*}
Observe that
\begin{equation*}
\begin{split}
\LL_hU_\alpha^{j}-\LL_hU_\gamma^{j}   &= \sum_{\beta\not=0}\left(J_p(U_{\alpha+\beta}^{j}-U_\alpha^{j}) -J_p(U_{\gamma+\beta}^{j}-U_\gamma^{j})  \right)\omega_\beta \\ 
&   =(p-1) \sum_{\beta\not=0}|\eta_\beta|^{p-2}( U_{\alpha+\beta}^{j}-U_{\gamma+\beta}^{j}) \omega_\beta- \left( U_{\alpha}^{j}-U_{\gamma}^{j}\right) (p-1)\sum_{\beta\not=0}|\eta_\beta|^{p-2}\omega_{\beta},
\end{split}
\end{equation*}
where $\eta_\beta$ lies between $U_{\alpha+\beta}^{j}-U_\alpha^{j}$ and $U_{\gamma+\beta}^{j}-U_\gamma^{j}$. Thus, proceeding as in \eqref{bound_yalpha},
\begin{equation*}
\begin{split}
 & \sum_{\beta\not=0}|\eta_\beta|^{p-2}\omega_{\beta}  \leq 2^{p-1}C_{s, p,d} (L_{u_0}+ T L_f)^{p-2}S_{a(p-2)}(r).
 \end{split}
\end{equation*}
Therefore, by \eqref{as:cfl} and \eqref{relation cfl tau}, and the induction hypothesis, we get

\begin{equation*}
\begin{split}
|U_\alpha^{j+1}-U_\gamma^{j+1}| &\leq |U_\alpha^{j}-U_\gamma^{j}|\left(1-\tau(p-1)\sum_{\beta\not=0}|\eta_\beta|^{p-2}\omega_{\beta}\right) \\ & \quad+ \tau(p-1) \sup_{\beta\in \Z^d}|U_{\alpha+\beta}^{j}-U_{\gamma+\beta}^{j}| \sum_{\beta\not=0}|\eta_\beta|^{p-2} \omega_\beta   +\tau|f_\alpha-f_\gamma| \\ 
&\leq (\Lambda_{u_0}(|x_\alpha-x_\gamma|) + t_j\Lambda_{f}(|x_\alpha-x_\gamma|)) \left(1-\tau(p-1)\sum_{\beta\not=0}|\eta_\beta|^{p-2}\omega_{\beta}\right) \\
&\quad + \tau(p-1) (\Lambda_{u_0}(|x_\alpha-x_\gamma|) + t_j\Lambda_{f}(|x_\alpha-x_\gamma|))  \sum_{\beta\not=0}|\eta_\beta|^{p-2} \omega_\beta + \tau\Lambda_f(|f_\alpha-f_\gamma|)\\
& =\Lambda_{u_0}(|x_\alpha-x_\gamma|) + t_{j+1}\Lambda_{f}(|x_\alpha-x_\gamma|),
\end{split}
\end{equation*}
which concludes the proof.
\end{proof}

We extend the scheme \eqref{scheme} continuously in space as follows:
\begin{equation}\label{scheme8}
\begin{cases} U^{j}(x)=U^{j-1}(x)+\tau \left(\LL_h U^{j-1}(x) +f(x)\right),  \quad x \in \mathbb{R}^{d}, j=1, ..., N,\\ 
U^{0}(x)=u_0(x), \quad  x \in \mathbb{R}^{d}.
\end{cases}
\end{equation}In the next result, we prove continuous dependence  of solutions to \eqref{scheme8} with respect to the initial data. 
\begin{lemma}\label{Equi 1}
Let $h,\tau,r\in(0,1)$ and $p>2$. Assume \eqref{as:weights}  and \eqref{as:cfl}.  Let $U$, $\tilde{U}$ be the solutions of \eqref{scheme8} corresponding to data $u_0, \tilde{u}_0$ and $f, \tilde{f}$, respectively, and satisfying \eqref{as:data}. Then, 
\begin{equation*}
\|U^{j}-\tilde{U}^{j}\|_{L^{\infty}(\R^d)} \leq \|u_0-\tilde{u}_0\|_{L^{\infty}(\R^d)}+ t_j\|f-\tilde{f}\|_{L^{\infty}(\R^d)},
\end{equation*}for $j=0, ..., N$. 
\end{lemma}
The proof follows exactly as in Lemma \ref{lem:equi}.
In the next result, we state two bounds for the variation in time of $U^{j}$: the first accounts for the equicontinuity in time of $U^{j}$ and the second will be used to state that the limiting profiles of $U^{j}$ are viscosity solutions of \eqref{parabolic problem}. 
\begin{lemma}\label{lem: holder comb}
Let $h,\tau,r\in(0,1)$ and $p>2$. Assume \eqref{as:weights}, \eqref{as:data} and \eqref{as:cfl}. Let $U$ be the solution of \eqref{scheme8}. Then, there exists a constant $K >0$ such that for all $k, j \geq 0$,
\begin{equation}\label{Eq in time}
\|U^{j+k}-U^{j}\|_{L^{\infty}(\R^d)} \leq \overline{\Lambda}_{u_0, f}(t_k),
\end{equation}where
$$\overline{\Lambda}_{u_0, f}(t_k):=\tilde{K}_1 t_k^{\frac{a}{2+(1-a)(p-2)}}+ \|f\|_{L^{\infty}(\R^d)}t_k,$$with $\tilde{K}_1:=KL_{u_0}^{\frac{p}{2+(1-a)(p-2)}} $. Moreover, there also holds
\begin{equation}\label{Eq in time s}
\|U^{j+k}-U^{j}\|_{L^{\infty}(\R^d)} \leq \tilde{\Lambda}_{u_0, f,r}(t_k),
\end{equation}
$$\tilde{\Lambda}_{u_0, f,r}(t_k):=\tilde{K}_2 t_k S_{a(p-1)}(r)+ \|f\|_{L^{\infty}(\R^d)}t_k,$$with $\tilde{K}_2:=KL_{u_0}^{p-1}C_{s, p, d}$. 
\end{lemma}
In the proof of this lemma we will use the following notation: for a given function $f$, we let
\begin{equation*}
d_yf(x):= f(x+y)-f(x) \quad \text{ and }\quad D_yf(x):= f(x+y)+f(x-y)-2f(x).
\end{equation*}
\begin{proof}
We consider a mollification $u_{0, \delta}$ of $u_0$ given by convolution with the standard mollifiers $\rho_\delta$, $\delta >0$. Let $(U_\delta)^{j}$ be the solution of \eqref{scheme8} with initial data $u_{0,\delta}$. Then,
\begin{equation}\label{Eq in time 2}
\|(U_\delta)^{1}- (U_\delta)^{0}\|_{L^{\infty}(\R^d)} \leq \tau \|\LL_h u_{0, \delta}\|_{L^{\infty}(\R^d)} +\tau \|f\|_{L^{\infty}(\R^d)}. 
\end{equation}Define $\tilde{U}_\delta^{j}:= U_\delta^{j+1}$. Hence, $\tilde{U}_\delta^{j}$ solves \eqref{scheme8} with initial condition $\tilde{U}_\delta^{0}:=U_\delta^{1}$. Consequently, by Lemma \ref{Equi 1} and \eqref{Eq in time 2},
$$\|U_\delta^{j+1}-U_\delta^{j}\|_{L^{\infty}(\R^d)} \leq \tau\|\LL_h u_{0, \delta}\|_{L^{\infty}(\R^d)} +\tau \|f\|_{L^{\infty}(\R^d)}.$$More general, for any $k \geq 0$ we have
\begin{equation*}
\|U_\delta^{j+k}-U_\delta^{j}\|_{L^{\infty}(\R^d)} \leq (k\tau)\|\LL_h u_{0, \delta}\|_{L^{\infty}(\R^d)} +(k\tau) \|f\|_{L^{\infty}(\R^d)}.
\end{equation*}Also, 

\begin{equation*}
\begin{split}
|\LL_h u_{0, \rho}(x)| \leq&  \frac{1}{2}  \sum_{0<|y_\beta|<1} (J_p(d_{y_\beta}u_{0, \delta}(x))- J_p(-d_{-y_\beta}u_{0, \delta}(x)) \omega_\beta + \sum_{|y_\beta|\geq 1}J_p(d_{y_\beta}u_{0, \delta}(x))\omega_\beta\\
\leq& C\left(\sum_{0<|y_\beta|<1} \max\left\lbrace |d_{y_\beta}u_{0, \delta}(x)|, |d_{-y_\beta}u_{0, \delta}(x)| \right\rbrace^{p-2} |D_{y_\beta}u_{0, \delta}(x)| \omega_\beta + \|u_{0,\delta}\|_{L^\infty(\R^d)}^{p-1} \sum_{|y_\beta|\geq 1}\omega_\beta\right).
\end{split}
\end{equation*}
We will prove first \eqref{Eq in time}.  Appealing to the H\"{o}lder regularity of $u_0$ and the following estimates  from \cite[Appendix A]{dTLi22b}
\begin{equation*}\label{est holder s1}
|d_{\pm y_{\beta}}u_{0, \delta}(x)| \leq KL_{u_0}\delta^{a-1}|y_\beta|
\end{equation*}
and
\begin{equation*}\label{est holder s12}
|D_{y_\beta}u_{0, \delta}(x)| \leq KL_{u_0}\delta^{a-2}|y_\beta|^{2}, \quad K> 0\end{equation*}
we obtain
\begin{equation*}
\begin{split}
|\LL_h u_{0,\delta}(x)|&\leq C\left( (KL_{u_0}\delta^{a-1})^{p-2}KL_{u_0}\delta^{a-2}\sum_{0<|y_\beta|<1}|y_\beta|^{p} \omega_\beta + \|u_{0}\|_{L^\infty(\R^d)}^{p-1} \sum_{|y_\beta|\geq 1}\omega_\beta \right)\\
&\leq C\left(K^{p-1}L_{u_0}^{p-1}\delta^{(a-1)(p-2)+a-2} +\|u_{0}\|_{L^\infty(\R^d)}^{p-1} \right) \\
&\leq CL_{u_0}^{p-1}\delta^{(a-1)(p-2)+a-2}.
\end{split}
\end{equation*}
Thus,
\[
\|U_\delta^{j+k}-U_\delta^{j}\|_{L^{\infty}(\R^d)} \leq  CL_{u_0}^{p-1}\delta^{(a-1)(p-2)+a-2} t_k + \|f\|_{L^\infty(\R^d)}t_k.
\]
Moreover,
\begin{equation}\label{ineq delta}
\begin{split}
\|U^{j+k}-U^{j}\|_{L^{\infty}(\R^d)} &\leq  \|U^{j+k}-U_\delta^{j+k}\|_{L^{\infty}(\R^d)}+\|U_\delta^{j+k}-U_\delta^{j}\|_{L^{\infty}(\R^d)} + \|U^{j}-U_\delta^{j}\|_{L^{\infty}(\R^d)} \\ & \leq  2L_{u_0}\delta^{a}+ CL_{u_0}^{p-1}\delta^{(a-1)(p-2)+a-2}t_k +\|f\|_{L^{\infty}(\R^d)} t_k.
\end{split}
\end{equation}Choosing $\delta =(\frac{C}{2L^{2-p}_{u_0}}t_k)^{\frac{1}{2+(1-a)(p-2)}}$ gives \eqref{Eq in time}.
To prove \eqref{Eq in time s}, we appeal to the H\"{o}lder regularity of $u_0$  to get
\begin{equation}\label{bound frac regularization}
\begin{split}
|\LL_h u_{0,\delta}|&\leq L^{p-1}_{u_0}\left(\sum_{0<|y_\beta|<1}|y_\beta|^{a(p-1)}\omega_\beta + \|u_{0}\|_{L^\infty(\R^d)}^{p-1} \sum_{|y_\beta|\geq 1}\omega_\beta \right)\\
& \leq 2L^{p-1}_{u_0}C_{s, p, d}S_{a(p-1)}(r).
\end{split}
\end{equation}

The rest of the proof follows as before by plugging \eqref{bound frac regularization} into  \eqref{ineq delta} and taking $\delta =(L^{p-2}_{u_0}C_{s,p, d}t_k)^{1/a}$ there.
\end{proof}

By a continuous interpolation, we extend the scheme \eqref{scheme8} in time in a continuous way as follows:
\begin{equation}\label{scheme11}
U(x,t):=\frac{t_{j+1}-t}{\tau}U^{j}(x)+\dfrac{t-t_j}{\tau}U^{j+1}(x), \quad \text{if }t \in  [t_j, t_{j+1}], \,\,\text{for some }j=0,..., N.
\end{equation}Observe that for all $t \in [t_j, t_{j+1}]$, there holds
\begin{equation}\label{eq-interp}
U(x, t)= U(x, t_j)+(t-t_j)\LL_hU(x, t_j) + (t-t_j)f(x).
\end{equation}Hence, the original scheme \eqref{scheme2} is preserved at any $(x, t)\in \overline{Q_T}$. 

By the same argument as in \cite[Proposition 3.10]{dTLi22b}, it follows for the solution $U$ of  \eqref{scheme11} that:
\begin{equation*}
\|U\|_{L^{\infty}( \overline{Q_T})} \leq \|u_0\|_{L^{\infty}(\R^d)} +T\|f\|_{L^{\infty}(\R^d)}
\end{equation*}and for any $z, x \in \mathbb{R}^{d}$ and $t, \tilde{t}\in [0, T]$, we have
\begin{equation*}
|U(x, t)-U(z,  \tilde{t})| \leq \Lambda_{u_0}(|x-z|)+T\Lambda_f(|x-z|) + C\overline{\Lambda}_{u_0, f}(|\tilde{t}-t|),
\end{equation*}where $\overline{\Lambda}_{u_0, f}$ comes from Lemma \ref{lem: holder comb}. By Arzel\`{a}-Ascoli theorem, we obtain the following convergence of the numerical solutions of the scheme \eqref{eq-interp}.
\begin{corollary}Assume \eqref{as:weights}, \eqref{as:data} and \eqref{as:cfl}. Let $U_h$ be a sequence of solutions of \eqref{eq-interp}. Then, there exists a subsequence $U_{h_i}$ and a function $u \in C_b( \overline{Q_T})$ such that
$$U_{h_i}\to u  \quad \text{locally uniformly in } \overline{Q_T} \text{ as }i \to \infty.$$
\end{corollary}

To finish the proof of Theorem \ref{Solutionparabolic prob}, we finally prove that the function $u$ is a viscosity solution. We recall this notion.
\begin{definition}\label{def:visc}
We say that a bounded and lower (resp., upper) semicontinuous function $u: \overline{Q}_T\to \mathbb{R}$ is a viscosity supersolution (resp., subsolution) of \eqref{parabolic problem} if
\begin{itemize}
\item[(a)] $u(x, 0)\geq u_0(x)$ (resp., $u(x, 0)\leq u_0(x)$), for all $x \in \mathbb{R}^{d}$;
\item[(b)] if $(x_0, t_0)\in Q_T$ and $\varphi \in C_b^{4}(B_R(x_0)\times (t_0-R, t_0+ R))\cap L^\infty(\overline{Q}_T)$, for some $R >0$ such that $\varphi(x_0, t_0)=u(x_0, t_0)$ and $\varphi(x, t) \leq u(x, t)$ (resp., $\varphi(x, t) \geq u(x, t)$ )  for all $(x, t) \in B_R(x_0)\times (t_0-R, t_0)$, then $$\partial_t \varphi(x_0, t_0) +(-\Delta)_p^{s}\varphi(x_0, t_0) \geq f(x_0) \quad (\text{resp., } \partial_t \varphi(x_0, t_0) +(-\Delta)_p^{s}\varphi(x_0, t_0) \leq f(x_0)).$$
\end{itemize}Finally, a viscosity solution is both, a viscosity supersolution and a viscosity subsolution.
\end{definition}

\begin{proof}[Proof of Theorem \ref{Solutionparabolic prob}] We follow the lines  of the proof of \cite[Theorem 4.1]{dTLi22b}. We start proving that $u$ is a viscosity supersolution. The proof that $u$ is a viscosity subsolution is similar. 

Let $\varphi$ be a smooth test function such that for some $(x^{*}, t^{*})$ there holds $u(x^{*}, t^{*})=\varphi(x^{*}, t^{*})$ and $u(x,t)> \varphi(x,t)$ for all $(x, t)\in B_R(x^{*})\times (t^{*}-R, t^{*}]$, $(x, t)\neq (x^{*}, t^{*})$.  The local uniform convergence of $U_h$ to $u$ implies that there is a sequence $(x^{h}, t^{h})$ converging to $(x^{*}, t^{*})$ such that
$$M_h:=\varphi(x^{h}, t^{h}) -U_h(x^{h}, t^{h}) \geq \varphi(x, t)-U_h(x, t) \quad \text{ for all }(x, t)\in B_R(x^{h}) \times (t^{h}-R, t^{h}].$$
By \eqref{eq-interp}, for $t_j  \in \mathcal{T}_\tau$,
\begin{equation}\label{eq-interph}
U_h(x^{h}, t^{h})= U_h(x^{h}, t_j)+(t^{h}-t_j)\sum_{\alpha \neq 0}J_p(U_h(x^{h}+ y_\alpha, t_j)-U_h(x^{h}, t_j))\omega_\alpha +(t^{h}-t_j)f(x^{h}).
\end{equation}
Define $\tilde{U}_h:=U_h+M_h$. Then $\tilde{U}_h$ also satisfies the scheme \eqref{eq-interph} and then, since $\tilde{U}_h(x^{h}, t^{h})= \varphi(x^{h}, t^{h})$, we get
\begin{equation*}
\varphi(x^{h}, t^{h}) = \tilde{U}_h(x^{h}, t_j) +(t^{h}-t_j)\sum_{\alpha \neq 0}J_p(\tilde{U}_h(x^{h}+ y_\alpha, t_j)-\tilde{U} _h(x^{h}, t_j))\omega_\alpha +(t^{h}-t_j)f(x^{h}).
\end{equation*}
Moreover, $\tilde{U}_h\geq \varphi$. 
Next, we will prove that we can replace $\tilde{U}_h$ by $\varphi$ in the above expression to get
\begin{equation}\label{varphi-visc}
\varphi(x^{h}, t^{h}) \geq \varphi(x^{h}, t_j) +(t^{h}-t_j)\sum_{\alpha \neq 0}J_p(\varphi(x^{h}+ y_\alpha, t_j)-\varphi _h(x^{h}, t_j))\omega_\alpha +(t^{h}-t_j)f(x^{h}).
\end{equation}
Then,  taking the limit as $h, \tau \to 0$ in \eqref{varphi-visc} and Theorem \ref{DiscretOrder} will end the proof. Thus, we introduce the function
$$g(\xi):=\xi +(t^{h}-t_j)\sum_{\alpha \neq 0}J_p(\tilde{U}_h(x^{h}+y_\alpha, t_j)-\xi)\omega_\alpha$$
and show that $g$ is non-decreasing. We will check that $g'(\xi) \geq 0$ for $ \xi \in [\varphi(x^{h}, t_{j}), \tilde{U}_h(x^{h}, t_j)].$ Observe that 
$$g'(\xi)= 1-(t^{h}-t_j)(p-1)\sum_{\alpha \neq 0}|\tilde{U}_h(x^{h}+y_\alpha, t_j)-\xi|^{p-2}\omega_\alpha.$$
Next,  for $ \xi \in [\varphi(x^{h}, t_{j}), \tilde{U}_h(x^{h}, t_j)]$ we have that

\begin{equation}\label{est visc 1}
\begin{split}
|\tilde{U}_h(x^{h}+y_\alpha, t_j)-\xi| &\leq |\tilde{U}_h(x^{h}+y_\alpha, t_j)-\tilde{U}_h(x^{h}, t_j)|+ |\tilde{U}_h(x^{h}, t_j)- \xi| \\ & \leq  |\tilde{U}_h(x^{h}+y_\alpha, t_j)-\tilde{U}_h(x^{h}, t_j)|+ |\tilde{U}_h(x^{h}, t_j)- \varphi(x^{h}, t_{j})|\\ & \leq |\tilde{U}_h(x^{h}+y_\alpha, t_j)-\tilde{U}_h(x^{h}, t_j)|+ |U_h(x^{h}, t_j)-U_h(x^{h}, t_h)|+ |\varphi(x^{h}, t^{h})-\varphi(x^{h}, t_j)| \\ & \leq \Lambda_{u_0}(|y_\alpha|)+T\Lambda_f(|y_\alpha|) +3\tilde{\Lambda}_{u_0, f,r}(\tau)+\tau \|\partial_t \varphi\|_{L^{\infty}(B_R(x^{h})\times [ t^{*}-R, t^{*}+R])},
\end{split}
\end{equation}where $\tilde{\Lambda}_{u_0, f,r}$ is from Lemma \ref{lem: holder comb} and we have used that  $|t^{h}-t_j|\leq \tau$. We now split the proof in two parts. If $a<sp/(p-1)$, i.e. $sp-a(p-1)=sp-a(p-2)-a>0$, then, by \eqref{as:cfl},
\begin{equation}\label{est visc 2}
\begin{split}
& 3\tilde{\Lambda}_{u_0, f,r}(\tau)+\tau \|\partial_t \varphi\|_{L^{\infty}(B_R(x^{h})\times [t^{*}-R, t^{*}+R])} \\ & \qquad  \leq \tau (3\tilde{K}_2 S_{a(p-1)}(r)  + 3\|f\|_{L^{\infty}(\R^d)}+ \|\partial_t \varphi\|_{L^{\infty}(B_R(x^{h})\times [t^{*}-R, t^{*}+R])})\\ 
&\qquad \leq r^{a} K_{s,p,d}  r^{sp-a(p-1)} (3\tilde{K}_2 S_{a(p-1)}(r)  + 3\|f\|_{L^{\infty}(\R^d)}+ \|\partial_t \varphi\|_{L^{\infty}(B_R(x^{h})\times [t^{*}-R, t^{*}+R])})\\
& \qquad \leq r^{a} \left({3\tilde{K}_2+1}\right).
\end{split}
\end{equation}
On the other hand, if $a\geq sp/(p-1)$, then
\begin{equation}\label{est visc 2b}
\begin{split}
& 3\tilde{\Lambda}_{u_0, f,r}(\tau)+\tau \|\partial_t \varphi\|_{L^{\infty}(B_R(x^{h})\times [t^{*}-R, t^{*}+R])} \\ & \qquad  \leq \tau (3\tilde{K}_2 S_{a(p-1)}(r)  + 3\|f\|_{L^{\infty}(\R^d)}+ \|\partial_t \varphi\|_{L^{\infty}(B_R(x^{h})\times [t^{*}-R, t^{*}+R])})\\ 
&\qquad \leq r^{a} K_{s,p,d}  \frac{(3\tilde{K}_2 S_{a(p-1)}(r)  + 3\|f\|_{L^{\infty}(\R^d)}+ \|\partial_t \varphi\|_{L^{\infty}(B_R(x^{h})\times [t^{*}-R, t^{*}+R])})}{|\log(r)|}\\
& \qquad \leq r^{a} \left({3\tilde{K}_2+1}\right).
\end{split}
\end{equation}
The above computations hold for all $r$ small enough and the constant $\tilde{K}_2$ comes from Lemma \ref{lem: holder comb}. Next, in the expression for $g'$ we will apply estimates \eqref{est visc 1} and \eqref{est visc 2} (or \eqref{est visc 2b}) for $|y_\alpha| < 1$, and we will bound $|\tilde{U}_h(x^{h}+y_\alpha, t_j)-\xi|^{p-2}$ uniformly  in the remaining region.  Consequently,

\begin{equation*}
\begin{split}
g'(\xi) &\geq 1- \tau(p-1)\bigg( \sum_{0 < |y_\alpha| < r}(\Lambda_{u_0}(|y_\alpha|)+T\Lambda_f(|y_\alpha|)+(3\tilde{K}_2+1 )r^{a})^{p-2}\omega_\alpha  \\ & \quad + \sum_{r \leq |y_\alpha| \leq 1} (\Lambda_{u_0}(|y_\alpha|)+T\Lambda_f(|y_\alpha|)+(3\tilde{K}_2+1 )r^{a})^{p-2}\omega_\alpha \\ & \quad+2^{p-2} \left(L_{u_0}+TL_f+1\right)^{p-2} \sum_{|y_\alpha| \geq 1}\omega_\alpha \bigg)\\
&\geq 1- \tau(p-1)  (L_{u_0}+TL_f + 3\tilde{K}_{2}+1)^{p-2} C_{s,p,d} \left(r^{a(p-2)-sp} + S_{a(p-2)}(r) +2^{p-2}\right)  \\
&\geq  1- \frac{1}{2^{p-1}} >0
\end{split}
\end{equation*}
where we have used \eqref{relation cfl tau}, the fact that $\tau r^{a(p-2)-sp} \leq  K_{s,d,p} $, and we have taken $r$ small enough. Hence, $g' \geq 0$ in  $[\varphi(x^{h}, t_j), \tilde{U}_h(x^{h}, t_j)]$. This ends the proof. 
\end{proof} 

\section*{Acknowledgements}

The three authors have been partially supported by Project SI3-PJI-2021-00324, cofunded by Universidad
Autónoma de Madrid and Comunidad de Madrid.

F. del Teso was supported by the Spanish Government through RYC2020-029589-I. and PID2021-127105NB-I00 funded by the MICIN/AEI.
Part of this material is based upon work supported by the Swedish Research Council under grant no. 2016-06596 while F. del Teso was in residence at Institut Mittag-Leffler in Djursholm,
Sweden, during the research program “Geometric Aspects of Nonlinear Partial Differential Equations”,
fall of 2022.

M. Medina has been partially supported by Project PDI2019-110712GB-100, MICINN,
Spain and by project RYC2020-030410-I. 

P. Ochoa has been partially supported by CONICET.

\bibliographystyle{abbrv}


\bibliography{Bibliography} 

\end{document}